\documentclass[11pt,reqno]{amsart}
\usepackage{}
\usepackage{a4wide}
\numberwithin{equation}{section}
\usepackage{mathrsfs}
\usepackage{amsfonts}
\usepackage{amsmath}
\usepackage{stmaryrd}
\usepackage{amssymb}
\usepackage{amsthm}
\usepackage{mathrsfs}
\usepackage{url}
\usepackage{amsfonts}
\usepackage{amscd}
\usepackage{indentfirst}
\usepackage{enumerate}
\usepackage{amsmath,amsfonts,amssymb,amsthm}
\usepackage{amsmath,amssymb,amsthm,amscd}
\usepackage{graphicx,mathrsfs}
\usepackage{appendix}
 \usepackage[numbers,sort&compress]{natbib}
\usepackage{color}
 \usepackage[colorlinks, linkcolor=blue, citecolor=blue]{hyperref}

\newcommand\N{\mathbb N}

\newcommand\R{\mathbb R}

\newcommand\mbb\mathbb
\newcommand\mbf\mathbf
\newcommand\mcal\mathcal
\newcommand\mfrak\mathfrak
\newcommand\mrm\mathrm
\newcommand\msf\mathsf
\renewcommand\a\alpha
\renewcommand\b\beta
\newcommand\g\gamma
\newcommand\G\Gamma
\renewcommand\d\delta
\newcommand\D\Delta
\newcommand\e\varepsilon
\newcommand\z\zeta
\renewcommand\t\theta
\newcommand\Th\Theta
\newcommand\la\lambda
\newcommand\La\Lambda
\newcommand\s\sigma
\newcommand\si\varsigma
\newcommand\Si\Sigma
\newcommand\ups\upsilon
\newcommand\U\Upsilon
\newcommand\ph\varphi
\renewcommand\o\omega
\renewcommand\O\Omega
\newcommand\wt\widetilde
\newcommand\wh\widehat
\newcommand\ol\overline
\newcommand\ul\underline
\newcommand\mr\mathring
\newcommand\ub\underbrace
\newcommand\pa\partial
\newcommand\n\nabla
\newcommand\fa\forall
\newcommand\ex\exists
\newcommand\es\emptyset
\newcommand\wk\rightharpoonup
\newcommand\inc\hookrightarrow
\newcommand\linf\varliminf
\newcommand\lsup\varlimsup
\newcommand\os\overset
\newcommand\us\underset
\newcommand\sr\stackrel
\newcommand\Ot\Leftarrow
\newcommand\To\Rightarrow
\newcommand\map\mapsto
\newcommand\ot\leftarrow
\newcommand\lot\longleftarrow
\newcommand\lto\longrightarrow
\newcommand\tot\leftrightarrow
\newcommand\ltot\longleftrightarrow
\newcommand\sm\backslash
\renewcommand\Cup\bigcup
\renewcommand\Cap\bigcap
\newcommand\sub\subset
\newcommand\Sub\Subset
\newcommand\sne\subsetneq
\newcommand\bus\supset
\newcommand\Bus\Supset

\newcommand\eq\equiv
\newcommand\ox\otimes
\newcommand\Ox\bigotimes
\newcommand\pl\oplus
\newcommand\Pl\bigoplus
\newcommand\x\times
\renewcommand\c\circ
\newcommand\q\quad
\renewcommand\l\left
\renewcommand\r\right
\newcommand\fr\frac

\newcommand{\var}{\varepsilon}

\newtheorem{Thm}{Theorem}[section]
\newtheorem{Lem}[Thm]{Lemma}

\newtheorem{Prop}[Thm]{Proposition}

\newtheorem{Rem}[Thm]{Remark}

\begin{document}

\title[Elliptic equations involving the critical Sobolev exponent]
{Concentrated solutions for a critical elliptic equation}
\author[L. Duan and S. Tian]{Lipeng Duan and Shuying Tian}

\address[Lipeng Duan]{School of Mathematics and Information Science,
Guangzhou University, Guangzhou 510006,  China.}
\email{lpduan777@sina.com}

\address[Shuying Tian] {Department of Mathematics, School of Science, Wuhan University of Technology, Wuhan 430070, China}
  \email{sytian@whut.edu.cn}
  \date{\today}
\begin{abstract}
In this paper, we are concerned with the following elliptic equation
\begin{equation*}
\begin{cases}
-\Delta u= Q(x)u^{2^*-1 }+\varepsilon
u^{s},~  &{\text{in}~\Omega},\\[1mm]
u>0,~  &{\text{in}~\Omega},\\[1mm]
u=0, &{\text{on}~\partial \Omega},
\end{cases}
\end{equation*}
where $N\geq 3$, $s\in [1,2^*-1)$ with $2^*=\frac{2N}{N-2}$, $\varepsilon>0$, $\Omega$ is a smooth bounded domain in $\mathbb{R}^N$.
 Under some conditions on $Q(x)$, Cao and Zhong in Nonlin. Anal. TMA (Vol 29, 1997, 461--483) proved that there exists a single-peak solution for small $\varepsilon$ if $N\geq 4$ and
$s\in (1,2^*-1)$. And they proposed in Remark 1.7 of their paper that

\vskip 0.1cm\begin{center}
\emph{``it is interesting to know the existence of single-peak solutions  for small $\varepsilon$  and  $s=1$''.}
\end{center}\vskip 0.1cm

\noindent  Also it was addressed in Remark 1.8 of their paper that

\vskip 0.1cm
\begin{center}
\emph{``the question of solutions concentrated at several points at the same time is still open''.}
\end{center}\vskip 0.1cm

\noindent
 Here we give some confirmative  answers to the above two questions. Furthermore, we prove the
local uniqueness of the multi-peak solutions. And our results show that
the concentration of the solutions to above problem is  delicate  whether $s=1$ or $s>1$.
\end{abstract}

\maketitle

{\small

\section{Introduction and main results}
\setcounter{equation}{0}
In this paper, we are concerned with the following  elliptic equations
\begin{equation}\label{1.1}
\begin{cases}
-\Delta u= Q(x)u^{2^*-1 }+\varepsilon
u^{s},~  &{\text{in}~\Omega},\\[1mm]
u>0,~  &{\text{in}~\Omega},\\[1mm]
u=0, &{\text{on}~\partial \Omega},
\end{cases}
\end{equation}
where $N\geq 3$, $s\in [1,2^*-1)$ with $2^*=\frac{2N}{N-2}$, $\varepsilon>0$ is a small parameter, $\Omega$ is a smooth and bounded domain in $\R^N$.
Let $0\leq Q(x)\in C(\bar\Omega)\bigcap C^2(\Omega)$ satisfy the following condition.

\vskip 0.2cm

\noindent \textbf{Condition (Q).} ~There exist $k$ different points $a_j = ({a_j^1}\cdots, {a_j^N} )  \in \Omega$  with $j=1,\cdots,k$ such that
\begin{equation}\label{4-24-1}
Q(a_j)>0,~\nabla Q(a_j)=0, ~ \Delta Q(a_j)<0~\mbox{and}~ ~det~\Big(\frac{\partial^2 Q(a_j)}{\partial x^i\partial x^l}   \Big)_{1\leq i,l\leq N} \neq 0,~\mbox{for}~j=1,\cdots,k.
\end{equation}

\vskip 0.2cm

For the case that $Q(x)$ is a positive constant, a well known fact is that  Brezis and Nirenberg \cite{BN1983} established the existence of one positive solution of \eqref{1.1} when $s=1$ and
  $\varepsilon \in (0,\lambda_1)$, where $\lambda_1$ denotes the first eigenvalue of $-\Delta$ with
0-Dirichlet boundary condition on $\partial\Omega$.
Also it is well known in \cite{Poh1965} that \eqref{1.1} has no solution in the case where $\Omega$ is star-shaped and  $\varepsilon=0$. From then a lot of attention has been paid to the limiting behavior  of the solutions $u_\varepsilon$
 of \eqref{1.1} as $\varepsilon\rightarrow 0$.
For more details on this aspect, one can refer to \cite{Gla1993,Han,MP2002,Rey1990}.
In this case, the positive solution will concentrate at the critical point of the Robin function or the Kirchhoff-Routh function. Also we point out that the research on the
 critical points of the Robin function or the Kirchhoff-Routh function  are very interesting, one can refer to very recently work \cite{Bartsch1} and the references
therein for further results.

When $Q(x)$ is not a positive constant, a very early result is \cite{Esc}. Here
the existence of problem \eqref{1.1} has been established when $s=1$ and
  $\varepsilon \in (0,\lambda_1)$.
A nature problem is that how about the property of concentrated solutions of problem \eqref{1.1} when $Q(x)$ is not a positive constant. In other words, if the concentrated solution of problem \eqref{1.1} exists, what is the location of the concentrated points. And under what conditions on $Q(x)$, can we find a concentrated solution of problem \eqref{1.1}.
 A partial answer of these questions is given by Cao and Zhong \cite{CZ1997}. Their results show  that the critical points of $Q(x)$
 play  a more important role than the critical points of the Robin function. And to state their results, we denote that $$
S=\inf \Big\{\displaystyle\int_{\Omega}|\nabla u|^2~~\big| ~~u\in H^1_0(\Omega),~\displaystyle\int_{\Omega}|u|^{2^*}=1\Big\}$$
is the best Sobolev constant and  $\delta_x$ is the Dirac mass at $x$.
Now Cao and Zhong's main results in \cite{CZ1997} are as follows.

\vskip 0.3cm

\noindent \textbf{Theorem A.}~\,~ Suppose that $N\geq 4$, $ s\in(1,\frac{N+2}{N-2})$ and $Q(x)$ satisfies Condition (Q). Then there exists an $\varepsilon_0>0$  such that for $\varepsilon\in (0,\varepsilon_0]$ and $i=1,\cdots,k$, problem \eqref{1.1} has a solution $u_\varepsilon$ satisfying (in the sense of measure)
\begin{equation*}
\big|\nabla u_\varepsilon\big|^2\rightharpoonup Q(a_i)^{-(N-2)/2}S^{N/2}\delta_{a_i}~~\mbox{as}~\varepsilon\rightarrow 0~~\mbox{and}~~
\big|  u_\varepsilon\big|^{2^*}\rightharpoonup Q(a_i)^{-N/2}S^{N/2}\delta_{a_i}~~\mbox{as}~\varepsilon\rightarrow 0.
\end{equation*}

 \vskip 0.2cm

Next, Cao and Zhong proposed the following two questions (Remark 1.7 and Remark 1.8 in \cite{CZ1997}):

\vskip 0.2cm

\noindent\emph{\textup{(i)}
\textbf{It is interesting to know the existence of single-peak solutions  for small $\varepsilon$  and  $s=1$.}}

\vskip 0.2cm

\noindent\emph{\textup{(ii)}
\textbf{The question of solutions which concentrate at several points at the same time is still open.}}

\vskip 0.2cm

\noindent And for further results on the concentrated solution of \eqref{1.1}, one can also refer to \cite{CN1995,CY1999}. In this paper, we give some confirmative  answers to the above questions by finite dimensional reduction and local Pohozaev identity. Also, we would like to point out that
local Pohozaev identity has been widely used on the concentrated solutions  in nonlinear elliptic equations recently(see \cite{Cao1,Deng,GMPY20,PWY2018}).

It is well-known that the equation $-\Delta u= u^{\frac{N+2}{N-2}} ~~\mbox{in}~\R^N$ has a family of solutions
\begin{equation*}
U_{x,\lambda}(y)=\big(N(N-2)\big)^{\frac{N-2}{4}}\frac{\lambda^{(N-2)/2}}{(1+\lambda^2|y-x|^2)^{(N-2)/2}},
\end{equation*}
where $x\in\R^N$ and $\lambda\in \R^+$.
Now for any given $f\in H^1(\Omega)$, let $P$ denote the projection from
$H^1(\Omega)$ onto $H^{1}_0(\Omega)$, i.e., $u=Pf$ is the solution of
\begin{equation*}
\begin{cases}
\Delta u=\Delta f, &{\text{in}~\Omega}, \\
u=0, &{\text{on}~\partial\Omega}.
\end{cases}
\end{equation*}

  First, we give the structure of the concentrated solutions of problem \eqref{1.1}.
\begin{Thm}\label{prop1}
Let $N\geq 4$, $s\in [1,2^*-1)$ and $m=1,\cdots,k$,
suppose that  $u_{\varepsilon}(x)$ is a solution of \eqref{1.1} satisfying
\begin{equation}\label{4-6-1}
\big|\nabla u_\varepsilon\big|^2\rightharpoonup \sum^m_{j=1}Q(a_j)^{-(N-2)/2}S^{N/2}\delta_{a_j}~~\mbox{as}~\varepsilon\rightarrow 0~~\mbox{and}~~
\big|  u_\varepsilon\big|^{2^*}\rightharpoonup \sum^m_{j=1} Q(a_j)^{-N/2}S^{N/2}\delta_{a_j}~~\mbox{as}~\varepsilon\rightarrow 0,
\end{equation}then
 $u_{\varepsilon}(x)$ can be written as
\begin{equation}\label{luo--2}
u_{\varepsilon}= \sum^m_{j=1}  Q(a_j)^{-(N-2)/4}    PU_{x_{\varepsilon,j},    \lambda_{\varepsilon,j}}+w_{\varepsilon},
\end{equation}
satisfying
 $\lambda_{\varepsilon,j}=:\big(u_\varepsilon(x_{\varepsilon,j})\big)^{\frac{2}{N-2}}$, $
 x_{\varepsilon,j} \rightarrow a_j,~\lambda_{\varepsilon,j}\rightarrow +\infty,~ \|w_{\varepsilon}\|=o(1)$.
\end{Thm}

Next result is concerned with the non-existence of concentrated solutions of \eqref{1.1} when $s=1$.
\begin{Thm}\label{th1-1}
Suppose that
$N\geq 5$, $s=1$ and $Q(x)$ satisfies Condition (Q). Then for any $m=1,\cdots,k$, problem \eqref{1.1} has no solutions $u_\varepsilon$ satisfying
 \eqref{4-6-1}.
\end{Thm}

Here our idea  to prove Theorem \ref{prop1} and Theorem \ref{th1-1} is as follows.
If $u_{\varepsilon}(x)$ is a solution of \eqref{1.1}, then we have following \emph{local Pohozaev identities}:
\emph{\begin{equation}\label{clp-1}
\frac{1}{2^*}\int_{\Omega'} \frac{\partial Q(x)}{\partial x^i} u_\varepsilon^{2^*}{\mathrm d}x =
\int_{\partial \Omega'}\Big(
\frac{\partial u_\varepsilon}{\partial \nu}\frac{\partial u_\varepsilon}{\partial x^i}
+\big(\frac{1}{2^*} Q(x) u_\varepsilon^{2^*} -\frac{1}{2} |\nabla u_\varepsilon|^2  + \frac{\varepsilon}{{s+1}}  u_\varepsilon^{s+1}\big)\nu^i   \Big){\mathrm d}\sigma,
\end{equation}
and
\begin{equation}\label{clp-10}
\begin{split}
\frac{1}{2^*}&\int_{\Omega'}\Big((x-x_{\varepsilon,j})\cdot\nabla Q(x)\Big)  u^{2^*}_\varepsilon   {\mathrm d}x +(1-\frac{N}{2}+\frac{N}{1+s})\varepsilon\int_{ \Omega'} u^{s+1}_\varepsilon {\mathrm d}x
\\=&
 \int_{\partial \Omega'} \Big[\Big(\frac{Q(x)}{2^*}u^{2^*}_\varepsilon+\frac{\varepsilon}{{s+1}
  } u_{\varepsilon}^{s+1}-\frac{1}{2}
|\nabla u_{\varepsilon}|^2\Big) \big((x-x_{\varepsilon,j})\cdot\nu\big)
  +
\Big((x-x_{\varepsilon,j})\cdot\nabla u_{\varepsilon}
+\frac{N-2}{2}u_{\varepsilon}\Big)\frac{\partial u_{\varepsilon}}{\partial\nu} \Big]{\mathrm d}\sigma,
\end{split}
\end{equation}
where $\Omega'\subset\subset \Omega$ and $\nu(x)=\big(\nu^{1}(x),\cdots,\nu^N(x)\big)$ is the outward unit normal of $\partial \Omega'$.}
Then using Pohozaev identity \eqref{clp-1} and blow-up analysis, we find the structure of the concentrated solutions (Theorem \ref{prop1}). Later we deduce Theorem  \ref{th1-1} by contradiction when $N\geq 5$
and  $s=1$ with the help of Pohozaev identity \eqref{clp-10} .

 \vskip 0.2cm

Now we tend to establish the existence of multi-peak solutions of \eqref{1.1} under some conditions.
\begin{Thm}\label{th1.2}
Suppose that $N=4$, $s=1$ or $N\geq 4$, $s\in (1,2^*-1)$, and $Q(x)$ satisfies Condition (Q). Then for any $m=1,\cdots,k$,  problem  \eqref{1.1} has a  solution  $u_\varepsilon$ satisfying \eqref{4-6-1}.
\end{Thm}
\noindent \emph{\textbf{Remark A. }
 Theorem \ref{th1-1}  and Theorem \ref{th1.2} show that
   problem \eqref{1.1} has no solutions satisfying  \eqref{4-6-1} when
         $N\geq 5$ and $s=1$,  while  problem \eqref{1.1} has a concentrated  solution satisfying  \eqref{4-6-1} when
         $N=4$ and $s=1$,  which
 answer  question (i) above(Remark 1.7 in \cite{CZ1997}).
 Taking $m\geq 2$ in Theorem \ref{th1.2}, then we find  that
     problem  \eqref{1.1} has a  solution concentrated at several points when $N=4$, $s=1$ or $N\geq 4$, $s>1$,
which  answers question (ii) above(Remark 1.8 in \cite{CZ1997}). In a word, our results show that
the concentration of the solutions of problem \eqref{1.1} is  delicate  whether $s=1$ or $s>1$.}

\vskip 0.2cm

Here  to prove Theorem \ref{th1.2}, the standard method is finite dimensional reduction. We define
 \begin{equation*}
 \begin{split}
 {\mathbb D}_{\varepsilon}=\Big\{(x_\varepsilon,\lambda_\varepsilon)| &~ x_\varepsilon=(x_{\varepsilon,1},\cdots,x_{\varepsilon,m}), ~~\lambda_\varepsilon=
 (\lambda_{\varepsilon,1},\cdots,\lambda_{\varepsilon,m}),  \\& ~|x_{\varepsilon,j}-a_j|=o(1),~ \frac{\lambda_{\varepsilon,j}}{\lambda_{\varepsilon,l}}\leq C~\mbox{and}~\lambda_{\varepsilon,j}\rightarrow +\infty, ~j,l=1,\cdots,m\Big\}.
 \end{split}\end{equation*}
 And for $(x_\varepsilon,\lambda_\varepsilon)\in {\mathbb D}_{\varepsilon}$,  by  Proposition \ref{Prop-Luo1} below which is standard, we find
that
$$u_\varepsilon=\displaystyle\sum^m_{j=1}\big(Q(a_j)\big)^{-\frac{N-2}{4}}PU_{x_{\varepsilon,j},
\lambda_{\varepsilon,j}}+  v_{\varepsilon}$$ solves
\begin{equation}\label{luou}
-\Delta u_\varepsilon= Q(x) u_\varepsilon^{\frac{N+2}{N-2}}+\varepsilon u^{s}_{\varepsilon}+\sum^m_{j=1}\sum^N_{i=0}c_{\varepsilon,i,j}
\varphi_{ij}(x),
\end{equation}
with $\varphi_{0j}(x)=\frac{\partial PU_{x_{\varepsilon,j},\lambda_{\varepsilon,j}}}{\partial \lambda}$ and $\varphi_{ij}(x)=\frac{\partial PU_{x_{\varepsilon,j},\lambda_{\varepsilon,j}}}{\partial x^{i}}$ for $j=1,\cdots,m$ and $i=1,\cdots,N$.

Next is to take suitable $(x_{\e},\lambda_{\e})$ such that $c_{\varepsilon,i,j}=0$   for $i=0,\cdots,N$ and $j=1,\cdots,m$. To do this, we use the following claim.

\vskip 0.2cm

\noindent \textbf{Claim 1:} \emph{  Suppose that $(x_\var,\lambda_\var)\in  {\mathbb D}_{\varepsilon}$ satisfies
\begin{equation}\label{a22-1-8}
\int_{\Omega}\Big(\Delta u_\varepsilon+Q(x) u_\varepsilon^{\frac{N+2}{N-2}}+\varepsilon u^{s}_{\varepsilon}   \Big)\frac{\partial PU_{x_{\varepsilon,j},\lambda_{\varepsilon,j}}}{\partial \lambda}=0,
\end{equation}
 and
\begin{equation}\label{a2-13-1}
\int_{\Omega}\Big(\Delta u_\varepsilon+Q(x) u_\varepsilon^{\frac{N+2}{N-2}}+\varepsilon u^{s}_{\varepsilon}   \Big) \frac{\partial PU_{x_{\varepsilon,j},\lambda_{\varepsilon,j}}}{\partial {x^i}} =0,
\end{equation}
then $c_{\varepsilon,i,j}=0$ for $i=0,\cdots,N$ and $j=1,\cdots,m$.}

Here we point out that the above modified finite dimensional reduction was used to find the
concentrated solutions of scalar curvature problem in \cite{PWY2018}.

 \vskip 0.2cm

Finally, the local uniqueness of the above multi-peak solutions can be stated as follows.
\begin{Thm}\label{th1.3}
Suppose that $N\geq 5$, $s\in (1,2^*-1)$ and $Q(x)$ satisfies Condition (Q). For any $m=1,\cdots,k$,  if $u^{(1)}_\varepsilon$ and $u^{(2)}_\varepsilon$ are the solutions of
problem  \eqref{1.1}  satisfying \eqref{4-6-1} and
\begin{equation}\label{lp111}
\frac{1}{C}\leq \frac{\lambda^{(1)}_{\varepsilon,j}}{\lambda^{(2)}_{\varepsilon,j}}\leq C, ~\mbox{for all}~
j=1,\cdots,m~\mbox{and some constant }~C>0,
\end{equation}
where $\lambda^{(l)}_{\varepsilon,j}$ ($l=1,2$ and $j=1,\cdots,m$) are the height of the $j$-th bubble of the solution $u^{(l)}_\varepsilon$ as in \eqref{luo--2},
 then $u^{(1)}_\varepsilon\equiv u^{(2)}_\varepsilon$ for $\varepsilon>0$ small.
\end{Thm}

Inspired by \cite{Cao1,Deng}, our idea to prove Theorem  \ref{th1.3} is as follows. Let $u^{(1)}_{\varepsilon}(x)$, $u^{(2)}_{\varepsilon}(x)$ be two different positive solutions concentrating at $\{a_1,\cdots, a_m\}$.
Set
\begin{flalign}\label{3.1}
\xi_{\varepsilon}(x)=\frac{u_{\varepsilon}^{(1)}(x)-u_{\varepsilon}^{(2)}(x)}
{\|u_{\varepsilon}^{(1)}-u_{\varepsilon}^{(2)}\|_{L^{\infty}(\Omega)}}.
\end{flalign}
Then we prove $
\xi_{\varepsilon}(x)=o(1)$ for $x\in \Omega$,
which is incompatible with the fact $\|\xi_{\varepsilon}\|_{L^{\infty}(\Omega)}=1$.
For the estimates near the critical  points, the following local Pohozaev  identities play a crucial role.

\vskip 0.2cm

\noindent \emph{If $u_{\varepsilon}^{(1)}(x)$ and $u_{\varepsilon}^{(2)}(x)$ are  solutions of \eqref{1.1}, then taking $u_\varepsilon=u_{\varepsilon}^{(l)}$ with $l=1,2$ in \eqref{clp-1} and \eqref{clp-10} with $\Omega'=B_d(x^{(1)}_{\e,j})$, and subtracting them on  both sides respectively, we can obtain
\begin{equation}\label{dclp-1}
\begin{split}
\int_{B_d(x^{(1)}_{\e,j})} &\frac{\partial Q(x)}{\partial x^i} D_{1,\varepsilon}\xi_\e {\mathrm d}x \\=&
 \int_{\partial B_d(x^{(1)}_{\e,j})}\Big( \big( Q(x) D_{1,\varepsilon}(x)+ \e D_{2,\varepsilon}(x)\big)\xi_\e -\frac{1}{2} \big(\nabla u^{(1)}_\varepsilon
 +\nabla u^{(2)}_\varepsilon\big)\cdot \nabla \xi_\e   \Big)\nu^i{\mathrm d}\sigma\\&
 +\int_{\partial B_d(x^{(1)}_{\e,j})}\Big(
\frac{\partial u^{(1)}_\varepsilon}{\partial \nu}\frac{\partial \xi_\e}{\partial x^i}
+\frac{\partial \xi_\e }{\partial \nu}\frac{\partial u^{(2)}_\varepsilon}{\partial x^i}\Big){\mathrm d}\sigma,
\end{split}\end{equation}
and
\begin{equation}\label{dclp-10}
\begin{split}
\int_{B_d(x^{(1)}_{\e,j})} & \Big((x-x^{(1)}_{\varepsilon,j}) \cdot \nabla Q(x)\Big)  D_{1,\varepsilon} \xi_\e {\mathrm d}x +\Big(\big(1+s\big)\big(1-\frac{N}{2}\big)+N\Big)\varepsilon\int_{B_d(x^{(1)}_{\e,j})} D_{2,\varepsilon}\xi_\e {\mathrm d}x
\\=&
\int_{\partial B_d(x^{(1)}_{\e,j})}
 \left(
 \Big( ( x-x^{(1)}_{\varepsilon,j}) \cdot  \nabla u^{(1)}_{\varepsilon}
+\frac{N-2}{2}
  u^{(1)}_{\varepsilon}  \Big)\frac{\partial \xi_{\varepsilon}}{\partial\nu}
  + \Big((x-x^{(1)}_{\varepsilon,j})  \cdot  \nabla \xi_{\varepsilon}
+\frac{N-2}{2}
  \xi_{\varepsilon}  \Big)\frac{\partial u^{(1)}_{\varepsilon}}{\partial\nu}
  \right) {\mathrm d}\sigma
\\&
 + \int_{\partial B_d(x^{(1)}_{\e,j})}\Big( \big( Q(x) D_{1,\varepsilon}(x)+ \e D_{2,\varepsilon}(x)\big)\xi_\e -\frac{1}{2} \big(\nabla u^{(1)}_\varepsilon
 +\nabla u^{(2)}_\varepsilon\big) \cdot\nabla \xi_\e \Big)  \Big( (  x-x^{(1)}_{\varepsilon,j} ) \cdot \nu\Big)
 {\mathrm d}\sigma,
\end{split}
\end{equation}
where
\begin{equation*}
D_{1,\varepsilon}(x)= \int_{0}^1
\Big(tu_{\varepsilon}^{(1)}(x)+(1-t)u_{\varepsilon}^{(2)}(x) \Big)
^{\frac{N+2}{N-2}}{\mathrm d}t~\mbox{and}~  D_{2,\varepsilon}(x)= \int_{0}^1
\Big(tu_{\varepsilon}^{(1)}(x)+(1-t)u_{\varepsilon}^{(2)}(x)\Big)
^{s}{\mathrm d}t.
\end{equation*}
}Here we would like to point out that the local Pohozaev identity \eqref{dclp-10} will have two terms involving
volume integral. Hence to calculate the two integrals precisely, we need to use some symmetries skillfully by some observations.

The paper is organized as follows.  In sections 2--5, we give the proof of theorems \ref{prop1}--\ref{th1.3} correspondingly.
Finally, we give the proofs of various local Pohozaev identities and some known facts in the Appendix.

Throughout this paper, we use the same $C$ to denote various generic positive constants independent with $\varepsilon$
and $\|\cdot\|$  to denote the basic norm in the Sobolev space $H^1_0(\Omega)$ and $\langle\cdot,\cdot\rangle$ to mean the corresponding inner product. And  we will use $D$ to denote the partial derivative for any function $h(y,x)$ with respect to $x$.

\section{The structure of the solutions (Proof of Theorem \ref{prop1})}
\setcounter{equation}{0}
\begin{Prop}
If $u_{\varepsilon}(x)$ is a solution of \eqref{1.1} with \eqref{4-6-1},
 then $|\nabla Q(a_j)|=0$ for $j=1,\cdots,m$.
\end{Prop}
\begin{proof}
We prove it by using local Pohozeav identity \eqref{clp-10}.
From \eqref{4-6-1}, we find
\begin{equation}\label{7-3-01}
 \int_{B_d(x_{\varepsilon,j})} \frac{\partial Q(x)}{\partial x^i} u_\varepsilon^{2^*}\rightarrow
 \frac{\partial Q(a_j)}{\partial x^i}Q(a_j)^{-N/2}S^{N/2} ~~ \mbox{as} ~~ \varepsilon\rightarrow 0.
\end{equation}
And using \eqref{4-6-1} again, we also know
\begin{equation}\label{7-3-02}
\int_{\partial   B_d(x_{\varepsilon,j} )   }\Big(
\frac{\partial u_\varepsilon}{\partial \nu}\frac{\partial u_\varepsilon}{\partial x^i}
+\big(\frac{1}{2^*} Q(x) u_\varepsilon^{2^*} -\frac{1}{2} |\nabla u_\varepsilon|^2  + \frac{\varepsilon}{{s+1}}  u_\varepsilon^{s+1}\big)\nu^i\Big){\mathrm d}\sigma \rightarrow 0~~ \mbox{as}~~\varepsilon\rightarrow 0.
\end{equation}
Then \eqref{clp-10}, \eqref{7-3-01} and \eqref{7-3-02} give us that  $|\nabla Q(a_j)|=0$  for $j=1,\cdots,m$.
\end{proof}

\begin{proof}[\textbf{Proof of Theorem \ref{prop1}}]
Since $u_{\varepsilon}(x)$ is a solution of \eqref{1.1} satisfying
\eqref{4-6-1}, we find that $u_{\varepsilon}(x)$ blows up at $a_1$. Then there exists $x_{1,\varepsilon}\in \Omega$  satisfying
$$x_{1,\varepsilon}\rightarrow a_1~\mbox{and}~u_\varepsilon(x_{1,\varepsilon})\rightarrow +\infty.$$
Let $v_{\varepsilon}=\lambda_{1,\varepsilon}^{-(N-2)/2}
u_{\varepsilon}\big(\frac{x}{\lambda_{1,\varepsilon}}+x_{1,\varepsilon}\big)$,
then
\begin{equation*}
-\Delta v_{\varepsilon} =Q\big(\frac{x}{\lambda_{1,\varepsilon}}+x_{1,\varepsilon}\big) v_{\varepsilon}^{2^*-1}
+\frac{   \lambda_{1,\varepsilon}^{ \frac{N-2}2 (s-1)  }  \varepsilon}{ \lambda_{1,\varepsilon}^2   }v^s_{\varepsilon},
~~ \mbox{in} ~~ \R^N.
\end{equation*}
For any fixed small $d$, $\displaystyle\max_{B_{d\lambda_{1,\varepsilon}}
(0)}v_{\varepsilon}=1$,
this means that
\begin{equation*}
u_{\varepsilon}=  Q(a_1)^{-(N-2)/4}  PU_{x_{1,\varepsilon},   \lambda_{1,\varepsilon}}+w_{1,\varepsilon}, ~~ \mbox{with}~~ \displaystyle\int_{B_d(x_{1,\varepsilon})}\big[|\nabla w_{1,\varepsilon}|^2+w^2_{1,\varepsilon}\big]=o(1),
\end{equation*}
and
\begin{equation*}
\big|\nabla w_{1,\varepsilon}\big|^2\rightharpoonup \sum^m_{j=2}Q(a_j)^{-(N-2)/2}S^{N/2}\delta_{a_j}  ~~ ~~ \mbox{and} ~~
\big|w_{1,\varepsilon}\big|^{2^*}\rightharpoonup \sum^m_{j=2} Q(a_j)^{-N/2}S^{N/2}\delta_{a_j}.
\end{equation*}
Repeating above step, we can find
\begin{equation*}
w_{1,\varepsilon}= Q(a_2)^{-(N-2)/4}  PU_{x_{2,\varepsilon},   \lambda_{2,\varepsilon}}+w_{2,\varepsilon}, ~~ \mbox{with} ~~\displaystyle\int_{B_d(x_{2,\varepsilon})}\big[|\nabla w_{2,\varepsilon}|^2+w^2_{2,\varepsilon}\big]=o(1).
\end{equation*}
 Then by finite step, we have
 \begin{equation}\label{4-24-2}
u_{\varepsilon}= \sum^m_{j=1}  Q(a_j)^{-(N-2)/4}    PU_{x_{\varepsilon,j},    \lambda_{\varepsilon,j}}+w_{\varepsilon}, ~~ \mbox{with} ~~ \big\|w_{\varepsilon}\big\|=o(1).
\end{equation}

\end{proof}

\begin{Rem}Now, for any $x\in\Omega$ and $\lambda\in \R^+$, we define
\begin{equation*}
\begin{split}
{\mathbb E}_{x,\lambda}=\left\{v\in H^1_0(\Omega)\big|~~ \Big\langle \frac{\partial PU_{x,\lambda}}{\partial \lambda},v \Big\rangle=\Big\langle \frac{\partial PU_{x,\lambda}}{\partial x^i},v\Big\rangle=0,~\mbox{for}~i=1,\cdots,N\right\}.
\end{split}
\end{equation*}
From \eqref{4-24-2}, we know
\begin{equation*}
 \Big\langle \frac{\partial PU_{x_{\varepsilon,j},\lambda_{\varepsilon,j}}}{\partial \lambda},w_{\varepsilon}\Big\rangle=\Big\langle \frac{\partial PU_{x_{\varepsilon,j},\lambda_{\varepsilon,j}}}{\partial x^i},w_{\varepsilon}\Big\rangle=o(1).
\end{equation*}
 Hence we can move $x_{\varepsilon,j}$ a bit (still denoted by $x_{\varepsilon,j}$), so that the above error term $w_{\varepsilon}\in \displaystyle\bigcap^m_{j=1}{\mathbb E}_{x_{\varepsilon,j},\lambda_{\varepsilon,j}}$.
\end{Rem}

\section{Non-existence (Proof of Theorem \ref{th1-1})}
\setcounter{equation}{0}
\subsection{Computations concerning $PU_{x_{\varepsilon,j},\lambda_{\varepsilon,j}}$}
\begin{Lem}\label{Lemma3.1}
For any small $d>0$, it holds
 \begin{equation}\label{3-13-07}
  \begin{split}
  \int_{B_d(x_{\varepsilon,j})}&
\Big( (x-x_{\varepsilon,j}) \cdot  \nabla Q(x)\Big) PU^{2^*}_{x_{\varepsilon,j},\lambda_{\varepsilon,j}}{\mathrm d}x=
  \frac{A \Delta Q(a_j)}{\lambda_{\varepsilon,j}^2}
  +o\Big(\frac{1}{\lambda_{\varepsilon,j}^2}   \Big)+
  O\Big(\frac{1}{\lambda_{\varepsilon,j}^{N-2}}   \Big),
  \end{split}
  \end{equation}
where $A=\frac{1}{N}\displaystyle\int_{\R^N}\frac{|y|^2}{(1+|y|^2)^N}{\mathrm d}y $. And
  \begin{equation}\label{3-13-08}
\begin{split}
 \int_{B_d(x_{\varepsilon,j})}PU_{x_{\varepsilon,j},\lambda_{\varepsilon,j}}^{s+1}{\mathrm d}x
 =&
\begin{cases}
\frac{B}{\lambda_{\varepsilon,j}^{2-\frac{N-2}{2}(s-1)}}
+o\Big(\frac{1}{\lambda_{\varepsilon,j}^{2-\frac{N-2}{2}(s-1)}}   \Big),&~\mbox{for}~N+s\neq 5,\\[6mm]
\omega_4\frac{\log \lambda_{\varepsilon,j}}{\lambda_{\varepsilon,j}^{2}}
+O\Big(\frac{1}{\lambda_{\varepsilon,j}^{2}}   \Big),&~\mbox{for}~N+s=5,
\end{cases}
\end{split}
\end{equation}
   where $B= \displaystyle\int_{\R^N}\frac{1}{(1+|y|^2)^{\frac{N-2}{2}(s+1)}}{\mathrm d}y $ and
   $\omega_4$ is a measure of the unit sphere of $\R^4$.
\end{Lem}
\begin{proof}[\underline{\textbf{Proof of \eqref{3-13-07}}}]
First, by \eqref{4-24-1}, \eqref{4-24-11} and \eqref{4-24-12}, we compute
  \begin{equation*}
  \begin{split}
  \int_{B_d(x_{\varepsilon,j})}&
\Big((x-x_{\varepsilon,j}) \cdot  \nabla Q(x)\Big)PU^{2^*}_{x_{\varepsilon,j},\lambda_{\varepsilon,j}} {\mathrm d}x\\=&
   \int_{B_d(x_{\varepsilon,j})}\Big((x-x_{\varepsilon,j}) \cdot  \nabla Q(x)\Big) U^{2^*}_{x_{\varepsilon,j},\lambda_{\varepsilon,j}}  {\mathrm d}x
  \\&+
  O\Big(\int_{B_d(x_{\varepsilon,j})}\big(U_{x_{\varepsilon,j},\lambda_{\varepsilon,j}}^{2^*-1}
  \varphi_{x_{\varepsilon,j},\lambda_{\varepsilon,j}}+
  \varphi_{x_{\varepsilon,j},\lambda_{\varepsilon,j}}^{2^*}\big)\cdot
  | x-x_{\varepsilon,j}|{\mathrm d}x   \Big)\\=&
   \int_{B_d(x_{\varepsilon,j})}\Big((x-x_{\varepsilon,j}) \cdot  \nabla Q(x)\Big) U^{2^*}_{x_{\varepsilon,j},\lambda_{\varepsilon,j}}  {\mathrm d}x
   +
  O\Big(\frac{1}{\lambda_{\varepsilon,j}^{N-2}} \Big).
  \end{split}
  \end{equation*}
  Also by basic  scaling transform, we find
  \begin{equation*}
  \begin{split}
   \int_{B_d(x_{\varepsilon,j})}&\Big((x-x_{\varepsilon,j}) \cdot  \nabla Q(x)\Big) U^{2^*}_{x_{\varepsilon,j},\lambda_{\varepsilon,j}}  {\mathrm d}x
  \\=&
  \int_{B_d(x_{\varepsilon,j})}
  \Big((x-x_{\varepsilon,j}) \cdot  \big ( \nabla Q(x) -   \nabla  Q(x_{\varepsilon,j} )\big) \Big)U_{x_{\varepsilon,j},\lambda_{\varepsilon,j}}^{2^*}  {\mathrm d}x\\=&
  \frac{\Delta Q(x_{\varepsilon,j})}{N}\int_{B_d(x_{\varepsilon,j})}U_{x_{\varepsilon,j},
  \lambda_{\varepsilon,j}}^{2^*}|x-x_{\varepsilon,j}|^2{\mathrm d}x
   +O\Big(\int_{B_d(x_{\varepsilon,j})}U_{x_{\varepsilon,j},\lambda_{\varepsilon,j}}^{2^*}
  |x-x_{\varepsilon,j}|^3{\mathrm d}x  \Big),
  \end{split}
  \end{equation*}
and
  \begin{equation*}
  \begin{split}
  \int_{B_d(x_{\varepsilon,j})}&U_{x_{\varepsilon,j},\lambda_{\varepsilon,j}}^{2^*}|x-x_{\varepsilon,j}|^2{\mathrm d}x
  \\=&
  \frac{1}{\lambda_{\varepsilon,j}^2}\int_{\R^N}\frac{|y|^2}{(1+|y|^2)^N}{\mathrm d}y
  -\frac{1}{\lambda_{\varepsilon,j}^2}\int_{\R^N\setminus
  B_{d\lambda_{\varepsilon,j}(0)}}\frac{|y|^2}{(1+|y|^2)^N}{\mathrm d}y
  \\=&
  \frac{1}{\lambda_{\varepsilon,j}^2}
  \int_{\R^N}\frac{|y|^2}{(1+|y|^2)^N}{\mathrm d}y +o\Big(\frac{1}{\lambda_{\varepsilon,j}^2}   \Big).
  \end{split}
  \end{equation*}
Similarly  it holds
  \begin{equation*}
   \int_{B_d(x_{\varepsilon,j})}U_{x_{\varepsilon,j},\lambda_{\varepsilon,j}}^{2^*}|x-x_{\varepsilon,j}|^3{\mathrm d}x
  =O\Big(\frac{1}{\lambda_{\varepsilon,j}^3}   \Big).
  \end{equation*}
The above calculations yield that
  \begin{equation*}
  \begin{split}
   \int_{B_d(x_{\varepsilon,j})}&
\Big(  (x-x_{\varepsilon,j}) \cdot  \nabla Q(x)\Big)U_{x_{\varepsilon,j},\lambda_{\varepsilon,j}}^{2^*} {\mathrm d}x
  \\=&\frac{1}{N\lambda_{\varepsilon,j}^2}\int_{\R^N}\frac{|y|^2}{(1+|y|^2)^N}{\mathrm d}y  \Delta Q(a_j)
  +o\Big(\frac{1}{\lambda_{\varepsilon,j}^2}   \Big)+
  O\Big(\frac{1}{\lambda_{\varepsilon,j}^{N-2}}   \Big).
  \end{split}\end{equation*}
Hence \eqref{3-13-07} follows by the above estimates.

\end{proof}

\begin{proof}[\underline{\textbf{Proof of \eqref{3-13-08}}}]
First, by \eqref{4-24-12}, we know
\begin{equation*}
\begin{split}
 \int_{B_d(x_{\varepsilon,j})}PU_{x_{\varepsilon,j},\lambda_{\varepsilon,j}}^{s+1}{\mathrm d}x
 =&
\int_{B_d(x_{\varepsilon,j})}U_{x_{\varepsilon,j},\lambda_{\varepsilon,j}}^{s+1}{\mathrm d}x
+
O\Big(\int_{B_d(x_{\varepsilon,j})}\big(U_{x_{\varepsilon,j},\lambda_{\varepsilon,j}}^s
|\varphi_{x_{\varepsilon,j},\lambda_{\varepsilon,j}}|+
|\varphi_{x_{\varepsilon,j},\lambda_{\varepsilon,j}}|^{s+1}\big){\mathrm d}x   \Big).
\end{split}
\end{equation*}
Also by scaling transform, we have
\begin{equation*}
\begin{split}
 \int_{B_d(x_{\varepsilon,j})}U_{x_{\varepsilon,j},\lambda_{\varepsilon,j}}^{s+1}{\mathrm d}x
=&
\frac{1}{\lambda_{\varepsilon,j}^{2-\frac{N-2}{2}(s-1)}}
\int_{B_{d\lambda_{\varepsilon,j}}(0)}\frac{1}{(1+|y|^2)^{\frac{N-2}{2}(s+1)}}{\mathrm d}y
\\=&
\begin{cases}
\frac{B}{\lambda_{\varepsilon,j}^{2-\frac{N-2}{2}(s-1)}}
+o\Big(\frac{1}{\lambda_{\varepsilon,j}^{2-\frac{N-2}{2}(s-1)}}   \Big),&~\mbox{for}~N+s\neq 5,\\[6mm]
\omega_4\frac{\log \lambda_{\varepsilon,j}}{\lambda_{\varepsilon,j}^{2}}
+O\Big(\frac{1}{\lambda_{\varepsilon,j}^{2}}   \Big),&~\mbox{for}~N+s=5.
\end{cases}
\end{split}
\end{equation*}
Next by \eqref{4-24-11} and H\"older's inequality, it follows
\begin{equation*}
\begin{split}
  \int_{B_d(x_{\varepsilon,j})}U_{x_{\varepsilon,j},\lambda_{\varepsilon,j}}^s
\varphi_{x_{\varepsilon,j},\lambda_{\varepsilon,j}}{\mathrm d}x
 =&O\left( \lambda_{\varepsilon,j}^{ \frac {N-2} {2}(s-1)}\int_{B_d(x_{\varepsilon,j})}\Big(\frac 1 {1+\lambda_{\varepsilon,j}^2|x-x_{\varepsilon,j}|^2}\Big)^{\frac{(N-2)s}{2}}\right)\\ =& O\Big(\frac{1}{\lambda_{\varepsilon,j}^{(N-2)s-\frac{N-2}{2}(s-1)}}\Big),
  \end{split}
 \end{equation*}
and
\begin{equation*}
\int_{B_d(x_{\varepsilon,j})}\varphi_{x_{\varepsilon,j},\lambda_{\varepsilon,j}}^{s+1}{\mathrm d}x
=
O\Big(\frac{1}{\lambda_{\varepsilon,j}^{\frac{N-2}{2}(s+1)}}   \Big).
\end{equation*}
Hence \eqref{3-13-08} holds by the above estimates.
\end{proof}
\begin{Lem}
It holds
\begin{equation}\label{4-24-14}
\left(\int_\Omega \Big(\sum^m_{j=1} \big|Q(x)-Q(a_j)\big|PU^{2^*-1}_{x_{\varepsilon,j},\lambda_{\varepsilon,j}}   \Big)^{\frac{2N}{N+2}}{\mathrm d}x\right)^{\frac{N+2}{2N}}=O\Big(\sum^m_{j=1}\big( \frac{1}{\lambda_{\varepsilon,j}^2}+|x_{\e,j}-a_j|^2\big)  \Big),
\end{equation}
and
\begin{equation}\label{4-24-15}
\begin{split}
\displaystyle\int_\Omega & \left(\big(\sum^m_{j=1}Q(a_j)^{-\frac{N-2}4}
PU_{x_{\varepsilon,j},\lambda_{\varepsilon,j}}\big)^{2^*-1}-\sum^m_{j=1}Q(a_j)^{-\frac{N+2}4} PU^{2^*-1}_{x_{\varepsilon,j},\lambda_{\varepsilon,j}}\Big)^{\frac{2N}{N+2}}{\mathrm d}x   \right)^{\frac{N+2}{2N}}\\=&
\begin{cases}
O\Big(\displaystyle\sum^m_{j=1} \frac{\big(\log \lambda_{\varepsilon,j}\big)^{\frac{N+2}{N}}}{\lambda_{\varepsilon,j}^{\frac{N+2}{2}}}  \Big),~&\mbox{if}~N\geq 6,\\
O\Big(\displaystyle\sum^m_{j=1}\frac{1}{\lambda_{\varepsilon,j}^{N-2}}\Big),~&\mbox{if}~N<6.
\end{cases}
\end{split}
\end{equation}
\end{Lem}
\begin{proof}
First, from  \eqref{4-24-1} and \eqref{4-24-11}, we know that
 \begin{equation*}
 \begin{split}
 \mbox{LHS of \eqref{4-24-14}}=&O\left(\sum^m_{j=1}\int_{B_\delta(x_{\varepsilon,j})}\Big(\big|x-a_j\big|^2   U^{2^*-1}_{x_{\varepsilon,j},\lambda_{\varepsilon,j}}\Big)^{\frac{2N}{N+2}}{\mathrm d}x\right)^{\frac{N+2}{2N}}\\=&O\left(\sum^m_{j=1}\int_{B_\delta(0)}
 \Big(|x_{\varepsilon,j}-a_j|^{\frac{4N}{N+2} }+|z|^{\frac{4N}{N+2}}\Big)\Big(\frac{\lambda_{\varepsilon,j}}{1+\lambda_{\varepsilon,j}^2|z|^2}\Big)^{N}              {\mathrm d}z\right)^{\frac{N+2}{2N}}\\=
&O\Big(\sum^m_{j=1}\big( \frac{1}{\lambda_{\varepsilon,j}^2}+|x_{\e,j}-a_j|^2\big)\Big).
 \end{split}
\end{equation*}
Next, from  \eqref{3-14-11}, we know that   when $N\ge 6$,
\begin{equation*}
\begin{split}
\mbox{LHS of \eqref{4-24-15}}=
&O\Bigg( \sum^m_{j=1} \sum^m_{i\ne j}\Big(\int_\Omega \big(PU_{x_{\varepsilon,j},\lambda_{\varepsilon,j}}   ^{\frac {2^*-1}{2}}PU_{x_{\varepsilon,i},\lambda_{\varepsilon,i}}^{\frac {2^*-1}{2}}\big)^{\frac{2N}{N+2}}   {\mathrm d}x \Big)^{\frac{N+2}{2N}}\Bigg)\\
=& O\Bigg(\sum^m_{j=1}\sum^m_{i\ne j}\Big(\int_{B_\delta(x_{\varepsilon,j})}\big(\frac{\lambda_{\varepsilon,j}}{1+\lambda_{\varepsilon,j}^2  |x-x_{\varepsilon,j}|^2}\big)^{\frac{N}{2}}\big(\frac{\lambda_{\varepsilon,i}}{1+\lambda_{\varepsilon,i}^2 |x-x_{\varepsilon,i}|^2}\big)^{\frac{N}{2}}{\mathrm d}x\Big)^{\frac{N+2}{2N}}\Bigg)
\\=&
O\Big(\displaystyle\sum^m_{j=1} \sum^m_{i\neq j}\frac{\big(\log \lambda_{\varepsilon,j}\big)^{\frac{N+2}{2N}}}{\lambda_{\varepsilon,i}^{\frac{N+2}{4}}
\lambda_{\varepsilon,j}^{\frac{N+2}{4}}}\Big)
=O\Big(\displaystyle\sum^m_{j=1} \frac{\big(\log \lambda_{\varepsilon,j}\big)^{\frac{N+2}{N}}}{\lambda_{\varepsilon,j}^{\frac{N+2}{2}}}\Big).
\end{split}
\end{equation*}
Similar,  when $N\le 6$, we have
\begin{equation*}
\begin{split}
\mbox{LHS of \eqref{4-24-15}}=&O\Bigg(\sum^m_{j=1}\sum^m_{i\ne j}\Big(\int_\Omega \big(PU_{x_{\varepsilon,j},\lambda_{\varepsilon,j}} ^{2^*-2 }        PU_{x_{\varepsilon,i},\lambda_{\varepsilon,i }}+            PU_{x_{\varepsilon,i},\lambda_{\varepsilon,i}}^{2^*-2}        PU_{x_{\varepsilon,j},\lambda_{\varepsilon,j}}\big )^{\frac{2N}{N+2}}{\mathrm d}x          \Big)^{\frac{N+2}{2N}}\Bigg) \\
 =&
 O\Bigg(\sum^m_{j =1}\sum^m_{i\ne j}\frac{1}{\lambda_{\varepsilon,i}^{\frac{N-2}{2}}}\Big(\int_{B_\delta  (x_{\varepsilon,j})}\Big(\frac{\lambda_{\varepsilon,j}}{1+\lambda_{\varepsilon,j}^2
 |x-x_{\varepsilon,j}|^2}\Big)^{\frac{4N}{N+2}}{\mathrm d}x\Big)^{\frac{N+2}{2N}}\Bigg)
 \\&
+ O\Bigg(\sum^m_{j=1}\sum^m_{i\ne j}\frac{1}{\lambda_{\varepsilon,i}^2}\Big(\int_{B_\delta(x_{\varepsilon,j})}\Big(\frac{
\lambda_{\varepsilon,j}}{1+\lambda_{\varepsilon,j}^2|x-x_{\varepsilon,j}|^2}\Big)^{\frac{N(N-2)}{N+2}} {\mathrm d}x\Big)^{\frac{N+2}{2N}}\Bigg)
\\
=&
 O\Bigg(\sum^m_{j=1}\sum^m_{i\ne j}\frac{1}{\lambda_{\varepsilon,i}^{\frac{N-2}{2}}}\frac{1}{\lambda_{\varepsilon,j}^{\frac{N-2}{2}}} \Bigg)+O\Bigg(\sum^m_{j=1}\sum^m_{i\ne j}\frac{1}{\lambda_{\varepsilon,i}^2}       \frac{1}{\lambda_{\varepsilon,j}^{\frac{N-2}{2}}}\Bigg)=O\Big(\displaystyle\sum^m_{j=1} \frac{1}{\lambda_{\varepsilon,j}^{N-2}}\Big).
\end{split}
\end{equation*}
Hence we complete \eqref{4-24-15} by above estimates.
\end{proof}

\subsection{Non-existence}
 \begin{Prop}
Let $u_{\varepsilon}(x)$ is a solution of \eqref{1.1}  with \eqref{4-6-1},
 then for small $d>0$, it holds
 \begin{equation}\label{ab4-18-12}
u_{\varepsilon}(x)=O\Big(\sum^m_{j=1}\frac{1}{\lambda^{(N-2)/2}_{\varepsilon,j}}\Big),
~\mbox{in}~C^1\Big(\Omega\backslash\bigcup^m_{j=1}B_d(x_{\varepsilon,j})   \Big).
\end{equation}
 \end{Prop}
\begin{proof}By potential theory,
we can find
 \begin{align}\label{4-24-29}
  u_{\varepsilon}(x)=\int_\Omega  G(x,y)\big( Q(x)u_{\varepsilon}^{2^*-1}+\varepsilon u_{\varepsilon}^{s}\big){\mathrm d}y
  =O\Big(\int_\Omega  \frac{1}{|x-y|^{N-2}}\big(u_{\varepsilon}^{2^*-1}+\varepsilon u_{\varepsilon}^{s}\big){\mathrm d}y\Big).
 \end{align}
 From \eqref{4-6-1}, we have $u_{\varepsilon}(x)=o(1)$ for $x\in \Omega\backslash\bigcup^m_{j=1}B_d(x_{\varepsilon,j})$ uniformly.
 And from the process of proofs of Theorem \ref{prop1} (by blow-up analysis), we know that
 $$u_\varepsilon(x)\leq C \sum^m_{j=1}PU_{x_{\varepsilon,j},\lambda_{\varepsilon,j}}~\mbox{in}~\bigcup^m_{j=1}B_d(x_{\varepsilon,j}).$$
 So for $ x\in \Omega\backslash\bigcup^m_{j=1}B_d(x_{\varepsilon,j})$, it holds
 \begin{equation*}
 \begin{split}
\int_\Omega & \frac{1}{|x-y|^{N-2}} u^{2^*-1}_{\varepsilon} {\mathrm d}y\\=&
\int_{\Omega\backslash\bigcup^m_{j=1}B_d(x_{\varepsilon,j})}  \frac{1}{|x-y|^{N-2}}  u^{2^*-1}_{\varepsilon} {\mathrm d}y+\underbrace{\int_{\bigcup^m_{j=1}B_d(x_{\varepsilon,j})}  \frac{1}{|x-y|^{N-2}}  u^{2^*-1}_{\varepsilon} {\mathrm d}y}_{:=I_1}\\ \leq &
\max_{\Omega\backslash\displaystyle\bigcup^m_{j=1}B_d(x_{\varepsilon,j})}|u_{\varepsilon}(x)|^{2^*-1}\int_{\Omega}  \frac{1}{|x-y|^{N-2}}   {\mathrm d}y+I_1
\leq C
\max_{\Omega\backslash\displaystyle\bigcup^m_{j=1}B_d(x_{\varepsilon,j})}|u_{\varepsilon}(x)|^{2^*-1} +I_1.
 \end{split}
 \end{equation*}
 Next, we find
 \begin{equation*}
 \begin{split}
 I_1=&\int_{\bigcup^m_{j=1}B_{\frac{d}{2}}(x_{\varepsilon,j})}  \frac{1}{|x-y|^{N-2}}  u^{2^*-1}_{\varepsilon} {\mathrm d}y+\int_{\bigcup^m_{j=1}\big(B_d(x_{\varepsilon,j})\backslash B_{\frac{d}{2}}(x_{\varepsilon,j})\big)}  \frac{1}{|x-y|^{N-2}}  u^{2^*-1}_{\varepsilon} {\mathrm d}y\\=&
 O\Big(\sum^m_{j=1}\int_{\Omega}PU^{2^*-1}_{x_{\varepsilon,j},\lambda_{\varepsilon,j}}\Big)
 + O\Big(\sum^m_{j=1} \frac{1}{\lambda^{\frac{N+2}{2}}_{\varepsilon,j}}\Big)=
 O\Big(\sum^m_{j=1} \frac{1}{\lambda^{\frac{N-2}{2}}_{\varepsilon,j}}\Big).
 \end{split}
 \end{equation*}
 These mean that
  \begin{equation}\label{4-24-30}
 \begin{split}
\int_\Omega & \frac{1}{|x-y|^{N-2}} u^{2^*-1}_{\varepsilon} {\mathrm d}y
=O\Big(
\max_{\Omega\backslash\displaystyle\bigcup^m_{j=1}B_d(x_{\varepsilon,j})}|u_{\varepsilon}(x)|^{2^*-1} +\sum^m_{j=1} \frac{1}{\lambda^{\frac{N-2}{2}}_{\varepsilon,j}}\Big).
 \end{split}
 \end{equation}
Similarly, we find
  \begin{equation}\label{4-24-31}
 \begin{split}
\int_\Omega   \frac{1}{|x-y|^{N-2}} u^{s}_{\varepsilon} {\mathrm d}y
 =&O\Big(
\max_{\Omega\backslash\displaystyle\bigcup^m_{j=1}B_d(x_{\varepsilon,j})}|u_{\varepsilon}(x)|^{s} +\sum^m_{j=1}\int_{\Omega}PU^{s}_{x_{\varepsilon,j},\lambda_{\varepsilon,j}}\Big)\\=&O\Big(
\max_{\Omega\backslash\displaystyle\bigcup^m_{j=1}B_d(x_{\varepsilon,j})}|u_{\varepsilon}(x)|^{s} +\sum^m_{j=1} \eta(\lambda_{\varepsilon,j})\Big),
 \end{split}
 \end{equation}
 where
\begin{equation}\label{ll1}
\eta\big({\lambda}\big)=
\begin{cases}
O\Big(\frac{1}{{\lambda}^{ N- \frac{   (N-2)s}{2}}} \Big),&~\mbox{for}~(N-2)s>N,\\[3mm]
O\Big(\frac{\log {\lambda}}{{\lambda}^{N-\frac{(N-2)s}{2}}} \Big),&~\mbox{for}~(N-2)s=N,\\[3mm]
O\Big(\frac{1}{{\lambda}^{ \frac{(N-2)s}{2}}} \Big),
&~\mbox{for}~(N-2)s<N.
\end{cases}
\end{equation}
Also we find
 \begin{equation}\label{4-24-32}
 \eta(\lambda_{\varepsilon,j})=O\Big(\frac{1}{\lambda^{\frac{N-2}{2}}_{\varepsilon,j}}\Big),~\mbox{for}~
 s\in \big[1,\frac{N+2}{N-2}\big).
 \end{equation}
 Hence  \eqref{4-24-29}, \eqref{4-24-30}, \eqref{4-24-31} and \eqref{4-24-32} give us that
  \begin{equation*}
u_{\varepsilon}(x)=O\Big(\sum^m_{j=1}\frac{1}{\lambda^{(N-2)/2}_{\varepsilon,j}}\Big),
~\mbox{in}~ \Omega\backslash\bigcup^m_{j=1}B_d(x_{\varepsilon,j}).
\end{equation*}

On the other hand,  for $x\in \Omega\backslash\displaystyle\bigcup^m_{j=1}B_{d}(x_{\varepsilon,j})$, we have
\begin{equation*}
\begin{split}
\frac{\partial u_\varepsilon(x)}{\partial x^i}=& \int_{\Omega}D_{x^i}G(y,x)\big( Q(x)u_{\varepsilon}^{2^*-1}+\varepsilon u_{\varepsilon}^{s}\big){\mathrm d}y
  =O\Big(\int_\Omega  \frac{1}{|x-y|^{N-1}}\big(u_{\varepsilon}^{2^*-1}+\varepsilon u_{\varepsilon}^{s}\big){\mathrm d}y\Big).
\end{split}
\end{equation*}
Repeating above estimates, we complete the proof of \eqref{ab4-18-12}.
\end{proof}
\begin{Prop}
Let $u_{\varepsilon}(x)$ is a solution of \eqref{1.1} with \eqref{luo--2},
 then it holds
 \begin{equation}\label{luo1}
 |x_{\e,j}-a_j|=o\Big(\sum^m_{l=1} \frac{1}{\lambda_{\varepsilon,l}}   \Big),~\mbox{for}~j=1,\cdots,m.
\end{equation}
 \end{Prop}
\begin{proof} We  using the identity \eqref{clp-1}   with $ \Omega' =B_d(x_{\e,j})$.
First, we have
\begin{equation}\label{4-24-23}
\begin{split}
 \mbox{LHS of}~\eqref{clp-1}=& \frac{1}{2^*}\int_{B_d(x_{\e,j})} \frac{\partial Q(x)}{\partial x^i} u_\varepsilon^{2^*}{\mathrm d}x \\=&
 \frac{1}{2^*}\sum^{N}_{l=1}\int_{B_d(x_{\e,j})} \frac{\partial^2 Q(a_j)}{\partial x^i\partial x^l} \big(x_l-{a_j^l}\big)u_\varepsilon^{2^*}{\mathrm d}x+O\Big(\int_{B_d(x_{\e,j})} |x-a_j|^2u_\varepsilon^{2^*}{\mathrm d}x   \Big)\\=&
 \frac{ S^{\frac N2} }{2^*}\big(Q(a_j)\big)^{-\frac{N}{2}}\sum^{N}_{l=1}  \frac{\partial^2 Q(a_j)}{\partial x^i\partial x^l} \big({x_{\e,j}^l} -{a_j^l}\big)+o \Big(\sum^m_{j=1}\big(|x_{\e,j}-a_j|+ \frac{1}{\lambda_{\varepsilon,j}}\big)   \Big).
\end{split}\end{equation}
On the other hand, using estimate \eqref{ab4-18-12}, we can get
\begin{equation}\label{4-24-24}
\begin{split}
 \mbox{RHS of}~\eqref{clp-1}=& O\Big(\sum^m_{j=1} \frac{1}{\lambda^{N-2}_{\varepsilon,j}}    \Big).
\end{split}\end{equation}
Then we find \eqref{luo1} by \eqref{4-24-23} and \eqref{4-24-24}.
\end{proof}
Let ${\bf Q}_\e$ be a quadratic form on ${\mathbb E}_{x_\varepsilon,\lambda_\varepsilon}:=\displaystyle\bigcap^m_{j=1}{\mathbb E}_{x_{\varepsilon,j},\lambda_{\varepsilon,j}}$ given by,
for any  $u,v \in {\mathbb E}_{x_\varepsilon,\lambda_\varepsilon}$,
\begin{equation}\label{Qe}
 \Big\langle{\bf Q}_\varepsilon u,v\Big\rangle_{{\e}}=\Big\langle u,v \Big\rangle-(2^*-1)\int_{\Omega}  Q(x)\Big(\sum^m_{j=1} Q(a_j)^{-(N-2)/4}PU_{x_{\varepsilon,j},\lambda_{\varepsilon,j}}\Big)^{2^*-2}uv.
\end{equation}
\begin{Prop}\label{Prop-Luo2}
For any $\varepsilon>0$ sufficiently small, there exists a constant $\rho>0$ such that

\begin{equation}\label{3-13-01}
\Big\langle{\bf Q}_\varepsilon v_\varepsilon,v_\varepsilon\Big\rangle_{\e} \geq \rho \|v_\varepsilon \|^2,
\end{equation}
where  $v_\varepsilon\in {\mathbb E}_{x_\varepsilon,\lambda_\varepsilon},~
|x_{\varepsilon,j}-a_j|=o(1)$ and
$\lambda_{\varepsilon,j}\rightarrow +\infty$, for $j=1,\cdots,m$.
 \end{Prop}
\begin{proof}  We prove it by contradiction.  Suppose that  the exist a sequence of $\e_n\to 0, |x_{{\e_n},j}-a_j|=o(1)$ and       $\lambda_{\varepsilon_n,j}\rightarrow +\infty$    such that   \[\Big\langle Q_{\e_n}v_{\e_n},v_{\e_n} \Big\rangle_{\e_n}\le\frac 1n\|v_{\e_n}\|^2.\]
Without loss of generality, we assume $\|v_{\e_n}\|=1$.  Then for $\psi\in H^1_0(\Omega)$,  we have
\begin{equation} \label{4191018}
\int_\Omega\nabla v_{\e_n}\nabla\psi-(2^*-1)\int_{\Omega}Q(x)\Big(\sum^m_{j=1}Q(a_j)^{-(N-2)/4} PU_{x_{\varepsilon,j},\lambda_{\varepsilon,j}}\Big)^{2^*-2}v_{\e_n}\psi=o(1).
\end{equation}
Now we define $\bar v_{\e_n}=\lambda_{1,{\e_n}}^{-(N-2)/2}v_{{\e_n}}\big(\frac{x}{\lambda_{1, {\e_n}}}+x_{1, {\e_n}}\big)$,
then for some universal constant $C$, it holds
$\displaystyle\int_{\R^N}|\nabla\bar v_{\e_n}|^2\le C$.
  So we have
\begin{equation*}
\bar{v}_{\e_n}\rightharpoonup\bar{v}~~\text{weakly in}~ H^1_{loc}(\mathbb{R}^N)~~\mbox{and}~
  \bar{v}_{\e_n}\rightarrow\overline{v}~~\text{strongly in}~ L^2_{loc}(\mathbb{R}^N).
 \end{equation*}
 From above and \eqref{4191018}, we can get $-\Delta v-(2^*-1)U_{0,1}^{2^*-2}v=0$.
So we can conclude that
\begin{equation*}
v=c_0\frac{\partial U_{0,\lambda}}{\partial\lambda}\Big|_{\lambda=1}+\sum_{i=1}^Nc_i\frac{\partial     U_{x,1}}{\partial x^i}\Big|_{x=0},~\text{for}~i=1,\cdots,N.
\end{equation*}
Also we have
 \begin{equation} \label{1054}
 \begin{split}
 &\int_{\Omega}Q(x)\Big(\sum^m_{j=1}Q(a_j)^{-(N-2)/4}PU_{x_{\varepsilon,j},\lambda_{\varepsilon,j}}  \Big)^{2^*-2}v_{\e_n}^2\\
 =&O\left(\int_{\bigcup^k_{j=1}B_d(x_{\e,j})}\Big(PU_{x_{\varepsilon,j},\lambda_{\varepsilon,j}}\Big)^{2^*-2}   v_{\e_n}^2+\int_{\Omega\setminus\bigcup^k_{j=1}B_d(x_{\e,j})}\Big(\sum^m_{j=1}PU_{x_{\varepsilon,j},
 \lambda_{\varepsilon,j}}\Big)^{2^*-2}v_{\e_n}^2\right)\\=&
 O\Big(\sum^m_{j=1}\frac{1}{\lambda_{\varepsilon,j}^{N-2}}\int_{\R^N}\frac{|v|}{(1+|x|^2)^2}\Big)
 +
 O\Big(\sum^m_{j=1}\frac{\|v_{\varepsilon_n}\|^2}{\lambda_{\varepsilon,j}^{2}} \Big)=
 o(1).
 \end{split}
 \end{equation}
 Combining \eqref{4191018} and \eqref{1054},  we conclude $\|v_{\e_n}\|=o(1)$,  which yields  a contradiction.
\end{proof}

Let $u_\varepsilon=\displaystyle\sum^m_{j=1} Q(a_j)^{-(N-2)/4}   PU_{x_{\varepsilon,j},\lambda_{\varepsilon,j}}+  w_{\varepsilon}$ be a solution of \eqref{1.1}, then
 \begin{equation}\label{2-16-1}
{\bf Q}_\varepsilon w_\varepsilon={\bf f}_\varepsilon+{\bf R}_\varepsilon(w_\varepsilon),
 \end{equation}
 where
\begin{equation}\label{fe}
{\bf f}_\varepsilon=Q(x) \Big(\sum^m_{j=1}Q(a_j)^{-\frac{N-2}{4}}   PU_{x_{\varepsilon,j},\lambda_{\varepsilon,j}}\Big)^{2^*-1}-\sum^m_{j=1}Q(a_j)^{-\frac{N-2}{4}}    PU_{x_{\varepsilon,j},\lambda_{\varepsilon,j}}^{2^*-1},
\end{equation}
and
\begin{equation}\label{Re}
\begin{split}
{\bf R}_\varepsilon(w_\varepsilon)=&Q(x)\Big(\big(\sum^m_{j=1}Q(a_j)^{-\frac{N-2}{4}}     PU_{x_{\varepsilon,j},\lambda_{\varepsilon,j}}+w_\varepsilon\big)^{2^*-1}-
\big(\sum^m_{j=1}Q(a_j)^{-\frac{N-2}{4}}PU_{x_{\varepsilon,j},\lambda_{\varepsilon,j}}\big)^{2^*-1}  \Big)\\&-(2^*-1)Q(x)\Big(\sum^m_{j=1}Q(a_j)^{-\frac{N-2}{4}}    PU_{x_{\varepsilon,j},\lambda_{\varepsilon,j}}\Big)^{2^*-2}
w_\varepsilon
+\varepsilon
\Big(\sum^m_{j=1}Q(a_j)^{-\frac{N-2}{4}}PU_{x_{\varepsilon,j},\lambda_{\varepsilon,j}}+w_\varepsilon  \Big)^{s}.
\end{split}\end{equation}
\begin{Lem}\label{lem1}
For $v\in H^1_0(\Omega)$, it holds
\begin{equation}\label{3-13-02}
\int_\Omega {\bf f}_\e v{\mathrm d}x =O\Big(\sum^m_{j=1}\big( \frac{1}{\lambda_{\varepsilon,j}^2}+|x_{\e,j}-a_j|^2\big)   \Big) \|v \|.
\end{equation}
\end{Lem}

\begin{proof}
First, by H\"older's inequality, we have
\begin{equation*}
\int_\Omega {\bf f}_\e v{\mathrm d}x =O\Big( \int_\Omega \big({\bf f}_\e\big)^{\frac{2N}{N+2}}{\mathrm d}x   \Big)^{\frac{N+2}{2N}}\|v \|.
\end{equation*}
Next, we compute
\begin{equation*}
\begin{split}
\Big(\int_\Omega \big({\bf f}_\e\big)^{\frac{2N}{N+2}}{\mathrm d}x    \Big)^{\frac{N+2}{2N}}=&
 O\left(\Big(\int_\Omega \Big(\sum^m_{j=1} \big|Q(x)-Q(a_j)\big|PU^{2^*-1}_{x_{\varepsilon,j},\lambda_{\varepsilon,j}}    \Big)^{\frac{2N}{N+2}}{\mathrm d}x\Big)^{\frac{N+2}{2N}}\right)
 \\&
+O\Bigg(\int_\Omega\Big(\big(\sum^m_{j=1} PU_{x_{\varepsilon,j},\lambda_{\varepsilon,j}}\big)^{2^*-1}
-\sum^m_{j=1}PU^{2^*-1}_{x_{\varepsilon,j},\lambda_{\varepsilon,j}}   \Big)^{\frac{2N}{N+2}}{\mathrm d}x\Big)^{\frac{N+2}{2N}} \Bigg).
\end{split}\end{equation*}
Hence using \eqref{4-24-14} and above estimates, we  get \eqref{3-13-02}.

\end{proof}

\begin{Lem}
For $v\in H^1_0(\Omega)$, it holds
\begin{equation}\label{dd3-13-02}
\int_\Omega {\bf R}_\e(w_\e) v{\mathrm d}x =o\Big( \|w_\e\|\Big) \|v \|.
\end{equation}
\end{Lem}
\begin{proof}
First, direct calculations yield that
\begin{equation}\label{3-13-03}
\begin{split}
\int_\Omega \Big(\sum^m_{j=1}PU_{x_{\varepsilon,j},\lambda_{\varepsilon,j}}\Big)^{s}v{\mathrm d}x=& O\left(\Big(\int_\Omega \Big(\sum^m_{j=1}PU_{x_{\varepsilon,j},\lambda_{\varepsilon,j}}  \Big)^{\frac{2Ns}{N+2}}{\mathrm d}x\Big)^{\frac{N+2}{2N}}\|v \|\right)\\=&O\Big(\sum^m_{j=1}\eta_1({\lambda}_{\varepsilon,j})\Big)\|v\|,
\end{split}
\end{equation}
where
\begin{equation*}
\eta_1(\lambda)=
\begin{cases}
\frac{1}{\lambda^{(N+2)/2-(N-2)s/2}}, & \mbox{if}~s>\frac{N+2}{2(N-2)},\\[2mm]
\frac{(\log \lambda)^{(N+2)/(2N)}}{\lambda^{(N+2)/2-(N-2)s/2}}, & \mbox{if}~s=\frac{N+2}{2(N-2)},\\[2mm]
\frac{1}{\lambda^{ (N-2)s/2}}, & \mbox{if}~s<\frac{N+2}{2(N-2)}.
\end{cases}
\end{equation*}
Since ${\bf R}_\e(w_\e)$ is the high order of $w_\e$, then \eqref{dd3-13-02} can be deduced  by
H\"older's inequality, \eqref{4-24-15}, \eqref{3-13-03} and \eqref{3-14-11}.
\end{proof}

 \begin{Prop}
Let $u_{\varepsilon}(x)$ is a solution of \eqref{1.1} with \eqref{4-6-1},
 then the error term $w_{\varepsilon}$ satisfies
 \begin{equation}\label{a4-18-12}
 \|w_{\varepsilon}\|=O\Big(\varepsilon \sum^m_{j=1}\eta_1(\lambda_{\varepsilon,j})\Big)
 +O\Big(\sum^m_{j=1}\frac{1}{\lambda_{\varepsilon,j}^2}\Big).
\end{equation}
 \end{Prop}
 \begin{proof}
 The estimate \eqref{a4-18-12} can be deduced  by \eqref{luo1}, \eqref{3-13-01}, \eqref{2-16-1}, \eqref{3-13-02} and \eqref{dd3-13-02}.
  \end{proof}

  \begin{Prop}
For any fixed small $d>0$, it holds
\begin{equation}\label{2020-01-02-1}
 \begin{split}                                                                                           \int_{B_d(x_{\varepsilon,j})}& \Big(( x-x_{\varepsilon,j}) \cdot \nabla Q(x)\Big)u_\varepsilon^{2^*}
 {\mathrm d}x   \\
=&A \big(Q(a_j)\big)^{-\frac{N}{2}} \frac{\Delta Q(a_j) }{\lambda_{\varepsilon,j}^2}
+o\Big(\sum^m_{j=1}\frac{1}{\lambda_{\varepsilon,j}^2}   \Big)+
O\Big(\sum^m_{j=1}\frac{1}{\lambda_{\varepsilon,j}^{N-2}}+\varepsilon \sum^m_{j=1}
\frac{\eta_1(\lambda_{\varepsilon,j})}{\lambda_{\varepsilon,j}} \Big),
\end{split}
\end{equation}
   where $A$ is the constant in Lemma \ref{Lemma3.1}.
\end{Prop}
\begin{proof}
  First, using \eqref{4-24-12}, we compute
  \begin{equation*}
  \begin{split}
   \int_{B_d(x_{\varepsilon,j})}&\Big(( x-x_{\varepsilon,j}) \cdot \nabla Q(x)\Big)u_\varepsilon^{2^*}
   {\mathrm d}x
  \\ =& \int_{B_d(x_{\varepsilon,j})}
  \Big(( x-x_{\varepsilon,j}) \cdot \nabla Q(x)\Big)
  \Big(\big(Q(a_j)\big)^{-\frac{N-2}{4}}PU_{x_{\varepsilon,j},
\lambda_{\varepsilon,j}}\Big)^{2^*}{\mathrm d}x
  \\ &+O\left(\int_{B_d(x_{\varepsilon,j})} \Big(   \big(Q(a_j)\big)^{-\frac{N-2}{4}}PU_{x_{\varepsilon,j},
\lambda_{\varepsilon,j}}\Big)^{2^*-1}|w_{\varepsilon}|\cdot |x-x_{\varepsilon,j}|{\mathrm d}x +
  \int_{B_d(x_{\varepsilon,j})}|w_{\varepsilon}|^{2^*}{\mathrm d}x\right).
  \end{split}
  \end{equation*}
  Also, by H\"older's inequality, we find
  \begin{equation*}\begin{split}
\int_{B_d(x_{\varepsilon,j})} & \Big(   \big(Q(a_j)\big)^{-\frac{N-2}{4}}PU_{x_{\varepsilon,j},
\lambda_{\varepsilon,j}}       \Big)^{2^*-1}
  |w_{\varepsilon}|\cdot |x-x_{\varepsilon,j}|    {\mathrm d}x
  \\=&
  O\Big(\big(\int_{B_d(x_{\varepsilon,j})}PU^{2^*}_{x_{\varepsilon,j},{\lambda_{\varepsilon,j}}}|x-x_{\varepsilon,j}|
  ^{\frac{2^*}{2^*-1}}{\mathrm d}x\big)^{\frac{2^*-1}{2^*}}\|w_{\varepsilon}\|_{L^{2^*}}   \Big)
  =
  O\Big(\frac{1}{\lambda_{\varepsilon,j}}\|w_{\varepsilon}\|   \Big).
  \end{split}
  \end{equation*}
  Then from above estimates, it holds
  \begin{equation*}
  \begin{split}
   &\int_{B_d(x_{\varepsilon,j})} \Big(( x-x_{\varepsilon,j}) \cdot \nabla Q(x)\Big)u_\varepsilon^{2^*}
   {\mathrm d}x \\
  =& \int_{B_d(x_{\varepsilon,j})}
  \Big(( x-x_{\varepsilon,j}) \cdot \nabla Q(x)\Big)
   \Big(\big(Q(a_j)\big)^{-\frac{N-2}{4}} PU_{x_{\varepsilon,j},
\lambda_{\varepsilon,j}}\Big)^{ 2^*}{\mathrm d}x
  + O\Big(\frac{\|w_{\varepsilon}\|}{\lambda_{\varepsilon,j}}+\|w_{\varepsilon}\|^{2^*}\Big),
  \end{split}
  \end{equation*}
 which, together with \eqref{3-13-07} and  \eqref{a4-18-12}, gives \eqref{2020-01-02-1}.
  \end{proof}

\begin{Lem}
It holds
\begin{equation}\label{2020-01-02-2}
   \begin{split}
   \displaystyle\int_{B_d(x_{\varepsilon,j})}u_\varepsilon^{s+1}{\mathrm d}x
              =&  \begin{cases}
    \frac{1}{\lambda_{\varepsilon,j}^{2-\frac{N-2}{2}(s-1)}}
                  \Big( \big(Q(a_j)\big)^{-\frac{(N-2)(s+1)}{4}} B
                 +o(1)   \Big),&~\mbox{for}~N+s\neq 5,\\[8mm]
                \frac{\omega_4}{Q(a_j)} \frac{\log \lambda_{\varepsilon,j}}{\lambda_{\varepsilon,j}^{2}}
                 +O(\frac{1}{\lambda_{\varepsilon,j}^{2}}),&~\mbox{for}~N+s=5,
                 \end{cases}
   \end{split}     \end{equation}
with $B$ and $\omega_4$ are the constants in Lemma \ref{Lemma3.1}.
\end{Lem}
\begin{proof}
First, using \eqref{4-24-12}, we have
\begin{equation}\label{Luo3.34}
\begin{split}
 \int_{B_d(x_{\varepsilon,j})}u_\varepsilon^{s+1}{\mathrm d}x
 =&
\int_{B_d(x_{\varepsilon,j})}  \Big(  \big(Q(a_j)\big)^{-\frac{N-2}{4}} PU_{x_{\varepsilon,j},\lambda_{\varepsilon,j}}        \Big)^{s+1}  {\mathrm d}x
\\   &~~ +O\left(
\int_{B_d(x_{\varepsilon,j})}\Big(\big(Q(a_j)\big)^{-\frac{(N-2)s}{4}} PU^{s}_{x_{\varepsilon,j},\lambda_{\varepsilon,j}} |w_{\varepsilon}|+
|w_{\varepsilon}|^{s+1}
\Big){\mathrm d}x   \right).
\end{split}
\end{equation}
 Furthermore, by H\"older's inequality and \eqref{a4-18-12}, we find
\begin{equation}\label{Luo3.35}
\begin{split}
\int_{B_d(x_{\varepsilon,j})}PU_{x_{\varepsilon,j},\lambda_{\varepsilon,j}}^{s}|w_{\varepsilon}|
 {\mathrm d}x
=&
O\Big(\big(\int_{B_d(x_{\varepsilon,j})}PU_{x_{\varepsilon,j},\lambda_{\varepsilon,j}}^{\frac{2Ns}{N+2}}
{\mathrm d}x\big)^{\frac{N+2}{2N}}\|w_{\varepsilon}\|   \Big)
=O\Big(\eta_1  \big(\lambda_{\varepsilon,j}\big)\|w_{\varepsilon}\|   \Big).
\end{split}
\end{equation}
Similarly, we know
\begin{equation}\label{Luo3.36}
\int_{B_d(x_{\varepsilon,j})}|w_{\varepsilon}|^{s+1}{\mathrm d}x
=O\Big(\|w_{\varepsilon}\|^{s+1}   \Big).
\end{equation}
Combining  Lemma \ref{Lemma3.1}, Lemma \ref{lem1}, \eqref{3-13-08}, \eqref{Luo3.34}, \eqref{Luo3.35} and \eqref{Luo3.36},  we get
  \eqref{2020-01-02-2}.

\end{proof}
\begin{proof}[\textbf{Proof of Theorem \ref{th1-1}:}]
If $u_{\varepsilon}(x)$ is a solution of \eqref{1.1}  with \eqref{4-6-1}, then we have
that $u_{\varepsilon}(x)$ has the structure  \eqref{luo--2}.
Also, Pohzaev identity  \eqref{clp-10} in domain $\Omega'=B_d(x_{\varepsilon,j})$  yields
  \begin{equation}\label{clp-11}
\begin{split}
\frac{1}{2^*}&\int_{ B_d(x_{\varepsilon,j})   }\Big( (  x-x_{\varepsilon,j}
) \cdot  \nabla Q(x) \Big)u^{2^*}_\varepsilon   {\mathrm d}x +(1-\frac{N}{2}+\frac{N}{1+s})\varepsilon\int_{ B_d(x_{\varepsilon,j})   } u^{s+1}_\varepsilon {\mathrm d}x
\\  =&
 \int_{\partial B_d(x_{\varepsilon,j})  } \Big[\Big(\frac{Q(x)}{2^*}u^{2^*}_\varepsilon+\frac{\varepsilon}{{s+1}
  } u_{\varepsilon}^{s+1}-\frac{1}{2}
|\nabla  u_{\varepsilon}|^2   \Big)   \Big( (  x-x_{\varepsilon,j}  ) \cdot \nu\Big)
  +
\Big(  (  x-x_{\varepsilon,j})\cdot  \nabla u_{\varepsilon}
+\frac{N-2}{2}
  u_{\varepsilon}   \Big)\frac{\partial u_{\varepsilon}}{\partial\nu} \Big]{\mathrm d}\sigma.
\end{split}
\end{equation}
  Then    from  \eqref{ab4-18-12},  we can find
   \begin{equation} \label{2020-01-02-3}
 \mbox{RHS of}~\eqref{clp-11}= O\Big(\sum^m_{j=1}\frac{1}{\lambda_{\varepsilon,j}^{{N-2}}}   \Big).
 \end{equation}
Taking $j_0\in \{1,\cdots,m\}$ such that $\frac{\lambda_{\varepsilon,j_0}}{\lambda_{\varepsilon,j}}\leq C$ for $j\neq j_0$ as $\varepsilon\to 0$. Here $C$ is independent with $\varepsilon$.
 If $N\geq 5$ and $s=1$, then \eqref{2020-01-02-1} gives
 \begin{equation}\label{2020-01-02-4}
  \begin{split}
  \int_{B_d(x_{\varepsilon,j_0})}&\Big(( x-x_{\varepsilon,j_0 }) \cdot \nabla Q(x) \Big)u_\varepsilon^{2^*}
  {\mathrm d}x
 = \big(Q(a_{j_0})\big)^{-\frac{N}{2}} \frac{\Delta Q(a_{j_0}) }{\lambda_{\varepsilon,j_0}^2}
\Big( A+o(1)\Big),
  \end{split}                                                                                             \end{equation}                                                                                           and from \eqref{2020-01-02-2}, we have
  \begin{equation}\label{2020-01-02-5}
   \displaystyle\int_{B_d(x_{\varepsilon,j_0})}u_\varepsilon^{s+1}{\mathrm d}x
=\big(Q(a_ {j_0} )\big)^{-\frac{(N-2)(s+1)}{4}} \frac{B+o(1)}{\lambda_{\varepsilon,j_0}^{2}}.
\end{equation}                                                                                                             Then   combining \eqref{clp-11}, \eqref{2020-01-02-3},  \eqref{2020-01-02-4}  and  \eqref{2020-01-02-5}, we can obtain
    \begin{equation*}
     \begin{split}
     \Delta Q(a_{j_0})
    = -\frac{B}{A} \big(Q(a_ {j_0} )\big)^{\frac{N}{2}-\frac{(N-2)(s+1)}{4}} \varepsilon
    +o\big(1\big)=o\big(1\big),
     \end{split}
     \end{equation*}
which is a  contradiction with \eqref{4-24-1}. Hence if $N\geq 5$, $s=1$ and $Q(x)$ satisfies Condition (Q), problem \eqref{1.1} has  no solutions $u_\varepsilon$ with \eqref{4-6-1}.
\end{proof}
\section{Existence (Proof of Theorem \ref{th1.2})}
\setcounter{equation}{0}
\subsection{Computations concerning $PU_{x_{\varepsilon,j},\lambda_{\varepsilon,j}}$}
\begin{Lem}
If $ N\geq 4 $,  we have
\begin{equation}\label{3-13-21}
\begin{split}
 \sum^m_{l=1} & \int_{\Omega} \big(Q(x)-Q(a_l) \big) PU_{x_{\varepsilon,l},\lambda_{\varepsilon,l}}^{\frac{N+2}{N-2}} \frac{\partial PU_{x_{\varepsilon,j},\lambda_{\varepsilon,j}}}{\partial\lambda }
= - \frac{\Delta Q(a_j)(N-2)A}{N\lambda^3_{\varepsilon,j}}  \big( 1+ o(1) \big)
  +o\Big(\sum^m_{l=1}\frac{1}{\lambda_{\varepsilon,l}^{N-1}} \Big).
  \end{split}
\end{equation}
\end{Lem}
\begin{proof}
First,  by scaling transform and \eqref{4-24-11}, we have, for $l=j$,
\begin{equation*}
\begin{split}
 \int_{\Omega}&\Big(Q(x)-Q(a_j) \Big)    PU_{x_{\varepsilon,j},\lambda_{\varepsilon,j }}^{\frac{N+2}{N-2}} \frac{\partial PU_{x_{\varepsilon,j},\lambda_{\varepsilon,j}}}{\partial
  \lambda }   \\=&
  \frac 1{2^*}  \int_{\Omega} \Big(Q(x)-Q(a_j) \Big)  \frac{\partial  U^{2^*}_{x_{\varepsilon,j},\lambda_{\varepsilon,j}}}{\partial
  \lambda } \\  &
    +
   O\Big(\frac{1}{\lambda_{\varepsilon,j}}\int_{\Omega} \big|Q(x)-Q(a_j)\big| \big( U_{x_{\varepsilon,j},\lambda_{\varepsilon,j}}^{\frac{N+2}{N-2}}
   \varphi_{x_{\varepsilon,j},\lambda_{\varepsilon,j}}+\varphi^{2^*}_{x_{\varepsilon,j},\lambda_{\varepsilon,j}}\big) \Big)\\ =&
  \frac{\Delta Q(a_j)}{ 2^* N}\int_{\Omega} \big|x-x_{\e,j}\big|^2\frac{\partial  U^{2^*}_{x_{\varepsilon,j},\lambda_{\varepsilon,j}}}{\partial
  \lambda}   \big( 1+ o(1) \big)
       +
   o\Big(\frac{1}{\lambda_{\varepsilon,j}^{N-1}} \Big)\\=&
 - \frac{\Delta Q(a_j)(N-2)A}{N\lambda^3_{\varepsilon,j}}  \big( 1+ o(1) \big)+  o\Big(\frac{1}{\lambda_{\varepsilon,j}^{N-1}} \Big).
\end{split}
\end{equation*}

On the other hand,  for $l\neq j$, we find
\begin{equation*}
\begin{split}
 \int_{\Omega}&\big(Q(x)-Q(a_l) \big) PU_{x_{\varepsilon,l},\lambda_{\varepsilon,l}}^{\frac{N+2}{N-2}} \frac{\partial PU_{x_{\varepsilon,j},\lambda_{\varepsilon,j}}}{\partial
  \lambda }\\ =&
\left\{    \int_{B_d(x_{\e,j})\bigcup B_d(x_{\e,l}) }    +  \int_{
   \Big( \Omega\backslash \big(B_d(x_{\e,j}) \bigcup B_d(x_{\e,l})\big) \Big)}   \right\}  \big(Q(x)-Q(a_l) \big) PU_{x_{\varepsilon,l},\lambda_{\varepsilon,l}}^{\frac{N+2}{N-2}} \frac{\partial PU_{x_{\varepsilon,j},\lambda_{\varepsilon,j}}}{\partial
  \lambda }\\ =&
  O\left(\frac{1}{\lambda_{\varepsilon,l}^{\frac{N+2}{2}}}\int_{B_d(x_{\e,j})}  \Big| \frac{\partial PU_{x_{\varepsilon,j},\lambda_{\varepsilon,j}}}{\partial
  \lambda }\Big|
  +\frac{1}{\lambda_{\varepsilon,j}^{\frac{N}{2}}} \int_{B_d(x_{\e,l}) } \Big|Q(x)-Q(a_l)\Big| PU_{x_{\varepsilon,l},\lambda_{\varepsilon,l}}^{\frac{N+2}{N-2}} + \frac{1}{\lambda_{\varepsilon,l}^{\frac{N+2}{2}}\lambda_{\varepsilon,j}^{\frac{N}{2}}} \right)\\ =&
  o\Big(\sum^m_{l=1}\frac{1}{\lambda_{\varepsilon,l}^{N-1}} \Big).
\end{split}
\end{equation*}
Then \eqref{3-13-21} can be deduced by above estimates.
\end{proof}
\begin{Lem}
It holds
\begin{equation}\label{3-13-31}
\begin{split}
   \sum^m_{l=1}  \int_{\Omega}&\Big(Q(x)-Q(a_l) \Big) PU_{x_{\varepsilon,l},\lambda_{\varepsilon,l}}^{\frac{N+2}{N-2}} \frac{\partial PU_{x_{\varepsilon,j},\lambda_{\varepsilon,j}}}{\partial {x^i}}
    \\ =&
\Big(2\int_{\R^N} \frac{|x|^2}{(1+|x|^2)^{N+1}}{\mathrm d}x+o(1)\Big)
    \sum^m_{l=1}\frac{\partial^2Q(a_j)}{\partial x^i\partial x^l}\big({x_{\e,j}^l} -{a_j^l} \big)  +
O\Big(\sum^m_{l=1}\frac{\log \lambda_{\e,l}}{\lambda^{N-1}_{\e,l}} \Big).
  \end{split}
\end{equation}
\end{Lem}
\begin{proof}
First, by scaling transform and \eqref{4-24-11}, we have
\begin{equation*}
\begin{split}
 \int_{\Omega}&\big(Q(x)-Q(a_j) \big) PU_{x_{\varepsilon,j},\lambda_{\varepsilon,j}}^{\frac{N+2}{N-2}} \frac{\partial PU_{x_{\varepsilon,j},\lambda_{\varepsilon,j}}}{\partial {x^i}}
      \\ =&
  \int_{\Omega} \big(Q(x)-Q(a_j) \big)  \frac{\partial  U^{2^*}_{x_{\varepsilon,j},\lambda_{\varepsilon,j}}}{\partial {x^i}}
 +
   O\Big( \lambda_{\varepsilon,j}\int_{\Omega} \big|Q(x)-Q(a_j)\big| \big( U_{x_{\varepsilon,j},\lambda_{\varepsilon,j}}^{\frac{N+2}{N-2}}
   \varphi_{x_{\varepsilon,j},\lambda_{\varepsilon,j}}+
   \varphi^{2^*}_{x_{\varepsilon,j},\lambda_{\varepsilon,j}}\big) \Big)\\ =&
 \frac{1}{2}\int_{\Omega} \sum^m_{l,q=1}\frac{\partial^2Q(a_j)}{\partial x^q\partial x^l}\big(x^l-{a_j^l}\big)\big(x^q-{a_{j}^q} \big)\frac{\partial  U^{2^*}_{x_{\varepsilon,j},\lambda_{\varepsilon,j}}}{\partial {x^i}}+
 O\Big(\int_{\Omega}  \big|x -a_{j}\big|^2 \Big|\frac{\partial  U^{2^*}_{x_{\varepsilon,j},\lambda_{\varepsilon,j}}}{\partial {x^i}}\Big| \Big)+
 O\Big(\frac{\log \lambda_{\varepsilon,j}}{\lambda_{\varepsilon,j}^{N-1} } \Big)
 \\ =&
\Big(2\int_{\R^N} \frac{|x|^2}{(1+|x|^2)^{N+1}}{\mathrm d}x+o(1) \Big)
    \sum^m_{l=1}\frac{\partial^2Q(a_j)}{\partial x^i\partial x^l}\Big({x_{\e,j}^l} -{a_j^l} \Big)  +
 O\Big(\frac{\log \lambda_{\varepsilon,j}}{\lambda_{\varepsilon,j}^{N-1} } \Big).
\end{split}
\end{equation*}
And for $l\neq j$, we find
\begin{equation*}
\begin{split}
 \int_{\Omega}&\Big(Q(x)-Q(a_l) \Big) PU_{x_{\varepsilon,l},\lambda_{\varepsilon,l}}^{\frac{N+2}{N-2}} \frac{\partial PU_{x_{\varepsilon,j},\lambda_{\varepsilon,j}}}{\partial {x^i}}\\ =&
   \int_{B_d(x_{\e,j}) \bigcup B_d(x_{\e,l}) \bigcup
   \Big( \Omega\backslash \big(B_d(x_{\e,j}) \bigcup B_d(x_{\e,l})\big) \Big)} \Big(Q(x)-Q(a_l) \Big) PU_{x_{\varepsilon,l},\lambda_{\varepsilon,l}}^{\frac{N+2}{N-2}} \frac{\partial PU_{x_{\varepsilon,j},\lambda_{\varepsilon,j}}}{\partial {x^i}}\\ =&
  O\left(\frac{1}{\lambda_{\varepsilon,l}^{\frac{N+2}{2}}}\int_{B_d(x_{\e,j})}  \Big| \frac{\partial PU_{x_{\varepsilon,j},\lambda_{\varepsilon,j}}}{\partial {x^i}}\Big|
  +\frac{1}{\lambda_{\varepsilon,j}^{\frac{N}{2}-2}} \int_{B_d(x_{\e,l}) } \Big|Q(x)-Q(a_l)\Big| PU_{x_{\varepsilon,l},\lambda_{\varepsilon,l}}^{\frac{N+2}{N-2}} + \frac{1}{\lambda_{\varepsilon,l}^{\frac{N+2}{2}}\lambda_{\varepsilon,j}^{\frac{N-4}{2}}} \right)\\ =&
 O\Big(\sum^m_{l=1}\frac{\log \lambda_{\varepsilon,l}}{\lambda_{\varepsilon,l}^{N-1} } \Big).
\end{split}
\end{equation*}
Then \eqref{3-13-31} can be deduced by above estimates.
\end{proof}

\begin{Lem}
 Suppose $ N\ge 4 $, then it holds
\begin{equation}\label{3-13-51}
  \begin{split}
     \sum^m_{l=1}\int_{\Omega}
     PU_{x_{\varepsilon,l},\lambda_{\varepsilon,l}}^{s}\frac{\partial PU_{x_{\varepsilon,j},\lambda_{\varepsilon,j}}}{\partial \lambda}
=& \begin{cases}
- \frac{4-(N-2)(s-1)}{2(s+1)}\Big(B+o(1)\Big)
\frac{1}{\lambda_{\varepsilon,j}^{3-\frac{N-2}{2}(s-1)}},&~\mbox{for}~N+s\neq 5,\\
- \omega_4\frac{\log \lambda_{\varepsilon,j}}{\lambda_{\varepsilon,j}^{3}}
+O\Big(\frac{1}{\lambda_{\varepsilon,j}^{3}}\Big),&~\mbox{for}~N+s=5,
\end{cases}
 \end{split}
\end{equation}
with $B$ and $\omega_4$ are the constants in Lemma \ref{Lemma3.1}.
And it holds
\begin{equation}\label{3-13-41}
  \begin{split}
     \sum^m_{l=1}\int_{\Omega}
     PU_{x_{\varepsilon,l},\lambda_{\varepsilon,l}}^{s}\frac{\partial PU_{x_{\varepsilon,j},\lambda_{\varepsilon,j}}}{\partial {x^i}}
=& O\Big(\sum^m_{l=1}\frac{1}{\lambda_{\varepsilon,l}^{N-2}} \Big). \end{split}
\end{equation}

\end{Lem}
\begin{proof}[\underline{\textbf{Proof of \eqref{3-13-51}}}]
 First,
for $l\neq j$, we have
\begin{equation}\label{a3-13-42}
  \begin{split}
     \int_{\Omega}&
     PU_{x_{\varepsilon,l},\lambda_{\varepsilon,l}}^{s}\frac{\partial PU_{x_{\varepsilon,j},\lambda_{\varepsilon,j}}}{\partial {\lambda}}
             \\=& \left\{\int_{B_d(x_{\e,j})}+\int_{ B_d(x_{\e,l})}+\int_{
 \Omega\backslash \big(B_d(x_{\e,j}) \bigcup B_d(x_{\e,l})\big) }\right\}      PU_{x_{\varepsilon,l},\lambda_{\varepsilon,l}}^{s} \frac{\partial PU_{x_{\varepsilon,j},\lambda_{\varepsilon,j}}}{\partial {\lambda}}
   \\ =&
   O\Big(\int_{B_d(x_{\e,j})}\big(\frac{\lambda_{\varepsilon,l}}{1+\lambda_{\varepsilon,l}^2|y -x_{\varepsilon,l}|^2}\big)^{\frac{(N-2)s}{2}}\frac{\lambda_{\varepsilon,j}^{\frac{N}{2}}|y -x_{\varepsilon,j}|^2}{\big(1+\lambda_{\varepsilon,j}^2|y-x_{\varepsilon,j}|^2\big)^{\frac{N}2}}\Big)  \\
  &+O\Big(\int_{B_d(x_{\e,l})}\big(\frac{\lambda_{\varepsilon,l}}{1+\lambda_{\varepsilon,l}^2|y-x_{\varepsilon,l}       |^2}\big)^{\frac{(N-2)s}{2}}\frac{\lambda_{\varepsilon,j}^{\frac{N}{2}}|y-x_{\varepsilon,j}|^2}{\big(1+   \lambda_{\varepsilon,j}^2|y-x_{\varepsilon,j}|^2\big)^{\frac N2}}\Big)
\\&+O\Big(\int_{ \Omega\backslash \big(B_d(x_{\e,j}) \bigcup B_d(x_{\e,l})\big) }\big(\frac{\lambda_{\varepsilon,l}}{1+\lambda_{\varepsilon,l}^2|y -x_{\varepsilon,l}|^2}\big)^{\frac{(N-2)s}{2}}\frac{\lambda_{\varepsilon,j}^{\frac{N}{2}}|y -x_{\varepsilon,j}|^2}{\big(1+\lambda_{\varepsilon,j}^2|y-x_{\varepsilon,j}|^2\big)^{\frac N2}}\Big)
 \\=&
 O\Big(\frac{1}{\lambda_{\e,l}^{\frac{(N-2)s}{2}}\lambda_{\e,j}^{\frac{N}{2}}}\Big)
 +O\Big(\frac{\eta(\lambda_{\e,l})}{\lambda_{\e,j}^{\frac{N}{2}}}\Big)
      = O\Big( \sum^m_{l=1}\frac{1}{ \lambda_{\varepsilon,l}^{N-1}} \Big),
  \end{split}
  \end{equation}
where $\eta(\lambda)$ is the function in \eqref{ll1}.
  On the other hand, we find
\begin{equation}\label{a3-13-43}
\begin{split}
     \int_{\Omega}&
     PU_{x_{\varepsilon,j},\lambda_{\varepsilon,j}}^{s}\frac{\partial PU_{x_{\varepsilon,j},\lambda_{\varepsilon,j}}}{\partial {\lambda}} \\=&\frac{1}{s+1} \frac{\partial }{\partial \lambda}
     \Big(\frac{1}{\lambda_{\varepsilon,j}^{2-\frac{N-2}{2}(s-1)}} \int_{B_{d\lambda_{\varepsilon,j}}(0)}   U^{s+1}_{x_{\varepsilon,j},\lambda_{\varepsilon,j}} \Big)
     \\=&\begin{cases}
- \frac{4-(N-2)(s-1)}{2(s+1)}\Big(B+o(1)\Big)
\frac{1}{\lambda_{\varepsilon,j}^{3-\frac{N-2}{2}(s-1)}},&~\mbox{for}~N+s\neq 5,\\
-\omega_4\frac{\log \lambda_{\varepsilon,j}}{\lambda_{\varepsilon,j}^{3}}
+O\Big(\frac{1}{\lambda_{\varepsilon,j}^{3}}\Big),&~\mbox{for}~N+s=5.
\end{cases}
\end{split}
\end{equation}
Then we find \eqref{3-13-51} by \eqref{a3-13-42} and \eqref{a3-13-43}.
\end{proof}
\begin{proof}[\underline{\textbf{Proof of  \eqref{3-13-41}}}]   First,
for $l\neq j$, we have
\begin{equation}\label{3-13-42}
  \begin{split}
     \int_{\Omega}&
     PU_{x_{\varepsilon,l},\lambda_{\varepsilon,l}}^{s}\frac{\partial PU_{x_{\varepsilon,j},\lambda_{\varepsilon,j}}}{\partial {x^i}}
             \\=& \left\{\int_{B_d(x_{\e,j})}+\int_{ B_d(x_{\e,l})}+\int_{
 \Omega\backslash \big(B_d(x_{\e,j}) \bigcup B_d(x_{\e,l})\big) }\right\}      PU_{x_{\varepsilon,l},\lambda_{\varepsilon,l}}^{s} \frac{\partial PU_{x_{\varepsilon,j},\lambda_{\varepsilon,j}}}{\partial {x^i}}
   \\ =&
   O\Big(\int_{B_d(x_{\e,j})}\big(\frac{\lambda_{\varepsilon,l}}{1+\lambda_{\varepsilon,l}^2|y -x_{\varepsilon,l}|^2}\big)^{\frac{(N-2)s}{2}}\frac{\lambda_{\varepsilon,j}^{\frac{N+2}{2}}|y -x_{\varepsilon,j}|}{\big(1+\lambda_{\varepsilon,j}^2|y-x_{\varepsilon,j}|^2\big)^{\frac{N}2}}\Big)  \\
  &+O\Big(\int_{B_d(x_{\e,l})}\big(\frac{\lambda_{\varepsilon,l}}{1+\lambda_{\varepsilon,l}^2|y-x_{\varepsilon,l}       |^2}\big)^{\frac{(N-2)s}{2}}\frac{\lambda_{\varepsilon,j}^{\frac{N+2}{2}}|y-x_{\varepsilon,j}|}{\big(1+   \lambda_{\varepsilon,j}^2|y-x_{\varepsilon,j}|^2\big)^{\frac N2}}\Big)
\\&+O\Big(\int_{ \Omega\backslash \big(B_d(x_{\e,j}) \bigcup B_d(x_{\e,l})\big) }\big(\frac{\lambda_{\varepsilon,l}}{1+\lambda_{\varepsilon,l}^2|y -x_{\varepsilon,l}|^2}\big)^{\frac{(N-2)s}{2}}\frac{\lambda_{\varepsilon,j}^{\frac{N+2}{2}}|y -x_{\varepsilon,j}|}{\big(1+\lambda_{\varepsilon,j}^2|y-x_{\varepsilon,j}|^2\big)^{\frac N2}}\Big)
 \\=&
 O\Big(\frac{1}{\lambda_{\e,l}^{\frac{(N-2)s}{2}}\lambda_{\e,j}^{\frac{N-2}{2}}}\Big)
 +O\Big(\frac{\eta(\lambda_{\e,l})}{\lambda_{\e,j}^{\frac{N-2}{2}}}\Big)
      = O\Big( \sum^m_{l=1}\frac{1}{ \lambda_{\varepsilon,l}^{N-2}} \Big),
  \end{split}
  \end{equation}
where $\eta(\lambda)$ is the function in \eqref{ll1}.
On the other hand, since $U_{0,1}(x)$ is an even function, we find
\begin{equation}\label{3-13-43}
\begin{split}
     \int_{\Omega}
     & PU_{x_{\varepsilon,j},\lambda_{\varepsilon,j}}^{s}\frac{\partial PU_{x_{\varepsilon,j},\lambda_{\varepsilon,j}}}{\partial {x^i}}=\frac{1}{s+1}
  \int_{\partial \Omega} PU_{x_{\varepsilon,j},\lambda_{\varepsilon,j}}^{s+1} \nu^i
   =O\Big( \frac{1}{\lambda_{\varepsilon,j}^{(N-2)(s+1)/2}} \Big).
\end{split}
\end{equation}
Then \eqref{3-13-41} can be deduced by \eqref{3-13-42} and \eqref{3-13-43}.

\end{proof}
\begin{Lem}
Suppose $N\ge 4$, we have the following estimates:
\begin{align}\label{4-26-2}
  \sum_{l\neq j} \int_{\Omega}&
 PU_{x_{\varepsilon,l},\lambda_{\varepsilon,l}}
 \Big|\frac{\partial PU^{2^* -1}_{x_{\varepsilon,j},\lambda_{\varepsilon,j}}}{\partial
  \lambda }\Big|= O\Big(\sum^m_{l=1}\frac{1}{\lambda_{\varepsilon,l}^{N-1}} \Big),
\end{align}
\begin{align}\label{a4-26-2}
  \sum_{l\neq j}   \int_{\Omega}&
 PU_{x_{\varepsilon,l},\lambda_{\varepsilon,l}}
 \Big|\frac{\partial PU^{s}_{x_{\varepsilon,j},\lambda_{\varepsilon,j}}}{\partial
  \lambda }\Big|=O\Big(\sum^m_{l=1}\frac{1}{\lambda_{\varepsilon,l}^{N-1}} \Big),
\end{align}
\begin{equation}\label{aa4-26-2}
\begin{split}
   \sum_{l\neq j}  \int_{\Omega}&
 PU_{x_{\varepsilon,l},\lambda_{\varepsilon,l}}
 \Big|\frac{\partial PU^{\frac{N+2}{N-2}}_{x_{\varepsilon,j},\lambda_{\varepsilon,j}}}{\partial
  {x^i}}\Big|=  O\Big(\sum^m_{l=1}\frac{1}{\lambda_{\varepsilon,l}^{N-3}} \Big),
\end{split}
\end{equation}
and
\begin{equation}\label{ab4-26-2}
\begin{split}
   \sum_{l\neq j}  \int_{\Omega}&
 PU_{x_{\varepsilon,l},\lambda_{\varepsilon,l}}
 \Big|\frac{\partial PU^{s}_{x_{\varepsilon,j},\lambda_{\varepsilon,j}}}{\partial {x^i}}\Big|=   O\Big(\sum^m_{l=1}\frac{1}{\lambda_{\varepsilon,l}^{N-3}} \Big).
\end{split}
\end{equation}
\end{Lem}
\begin{proof}
First, we have
\begin{equation}\label{3-14-01}
\begin{split}
   \int_{\Omega}&
 PU_{x_{\varepsilon,l},\lambda_{\varepsilon,l}}
 \Big|\frac{\partial PU^{ 2^* -1}_{x_{\varepsilon,j},\lambda_{\varepsilon,j}}}{\partial
  \lambda }\Big| \\ =&
 O\Big( \int_{\Omega}
 \frac{\lambda_{\e,l}^{\frac{N-2}{2}}}{\big(1+\lambda_{\e,l}^2|x-x_{\e,l}|^2\big)^{\frac{N-2}{2}}}
 \cdot   \frac{\lambda_{\e,j}^{\frac{N+4}{2}}     |x-x_{\e,j}|^2}{\big(1+\lambda_{\e,j}^2|x-x_{\e,j}|^2\big)^{\frac{N+4}{2}}}
   \Big)\\ =&
 O\left( \int_{B_d(x_{\e,j})}\frac{1}{ \lambda_{\e,l}^{\frac{N-2}{2}}}
 \frac{\lambda_{\e,j}^{\frac{N+4}{2}}|x-x_{\e,j}|^2}{\big(1+\lambda_{\e,j}^2
 |x-x_{\e,j}|^2\big)^{\frac{N+4}{2}}}
\right)+O\Big(\frac{1}{ \lambda_{\e,l}^{\frac{N-2}{2}}\lambda_{\e,j}^{\frac{N+4}{2}}} \Big)\\ &
+O\left(
  \frac{1}{ \lambda_{\e,j}^{\frac{N+4}{2}}} \int_{B_d(x_{\e,l})}
 \frac{\lambda_{\e,l}^{\frac{N-2}{2}}}{\big(1+\lambda_{\e,l}^2|x-x_{\e,l}|^2\big)^{\frac{N-2}{2}}}
\right)
\\
  &  =   O\Big(\frac{1}{ \lambda_{\e,l}^{\frac{N-2}{2}}\lambda_{\e,j}^{\frac{N}{2}}} \Big) +  O\Big(\frac{1}{ \lambda_{\e,l}^{\frac{N-2}{2}}\lambda_{\e,j}^{\frac{N+4}{2}}} \Big)       =
   O\Big(\sum^m_{l=1}\frac{1}{\lambda_{\varepsilon,l}^{N-1}} \Big).
\end{split}
\end{equation}
Similar to \eqref{3-14-01}, we find
\begin{equation*}
\begin{split}
   \int_{\Omega}&
 PU_{x_{\varepsilon,l},\lambda_{\varepsilon,l}}\Big|
 \frac{\partial PU^{s}_{x_{\varepsilon,j},\lambda_{\varepsilon,j}}}{\partial \lambda }\Big|\\ =&
 O\Big( \int_{\Omega}
 \frac{\lambda_{\e,l}^{\frac{N-2}{2}}}{\big(1+\lambda_{\e,l}^2|x-x_{\e,l}|^2\big)^{\frac{N-2}{2}}}
 \cdot \frac{\lambda_{\e,j}^{\frac{(N-2)s}{2}+1}}{\big(1+\lambda_{\e,j}^2|x-x_{\e,j}|^2\big)^{\frac{(N-2)s}{2}+1}}
 |x-x_{\e,j}|^2
   \Big)\\
  =&O\left( \frac{1}{ \lambda_{\e,l}^{\frac{N-2}{2}}\lambda_{\e,j}}\eta_2(\lambda_{\e,j})
 +
  \frac{1}{ \lambda_{\e,j}^{\frac{(N-2)s}{2}+1}\lambda_{\e,l}^{\frac{N-2}{2}}} +\frac{1}{ \lambda_{\e,l}^{\frac{N-2}{2}} \lambda_{\e,j}^{\frac{(N-2)s}{2}+1}}\right)
\\
  =&O\left( \sum^m_{l=1}\frac{\eta(\lambda_{\e,l})}{ \lambda_{\e,l}^{\frac{N}{2}}}
 + \sum^m_{l=1}\frac{1}{\lambda_{\e,l}^{\frac{(N-2)(s+1)}{2}+1}} \right) =
   O\Big(\sum^m_{l=1}\frac{1}{\lambda_{\varepsilon,l}^{N-1}} \Big),
\end{split}
\end{equation*}
where $\eta(\lambda)$ is the function in \eqref{ll1}.
Also, we compute that
\begin{equation*}
\begin{split}
   \int_{\Omega}&
 PU_{x_{\varepsilon,l},\lambda_{\varepsilon,l}}
 \Big|\frac{\partial PU^{\frac{N+2}{N-2}}_{x_{\varepsilon,j},\lambda_{\varepsilon,j}}}{\partial
  {x^i}}\Big|    \\ =&
 O\Big( \int_{\Omega}
 \frac{\lambda_{\e,l}^{\frac{N-2}{2}}}{\big(1+\lambda_{\e,l}^2|x-x_{\e,l}|^2\big)^{\frac{N-2}{2}}}
 \cdot \frac{\lambda_{\e,j}^{\frac{N+6}{2}}}{\big(1+\lambda_{\e,j}^2|x-x_{\e,j}|^2\big)^{\frac{N+4}{2}}}
 |x-x_{\e,j}|
   \Big)    \\ =&
 O\left( \int_{B_d(x_{\e,j})}\frac{1}{ \lambda_{\e,l}^{\frac{N-2}{2}}}
 \frac{\lambda_{\e,j}^{\frac{N+6}{2}}|x-x_{\e,j}|}{\big(1+\lambda_{\e,j}^2
 |x-x_{\e,j}|^2\big)^{\frac{N+4}{2}}}
\right)
+O\big(\frac{1}{ \lambda_{\e,l}^{\frac{N-2}{2}}\lambda_{\e,j}^{\frac{N+2}{2}}} \big)    \\ &
+ O\left(
  \frac{1}{ \lambda_{\e,j}^{\frac{N+2}{2}}} \int_{B_d(x_{\e,l})}
 \frac{\lambda_{\e,l}^{\frac{N-2}{2}}}{\big(1+\lambda_{\e,l}^2|x-x_{\e,l}|^2\big)^{\frac{N-2}{2}}}
\right) = O\big(\sum^m_{l=1}\frac{1}{\lambda_{\varepsilon,l}^{N-3}} \big).
\end{split}
\end{equation*}
And it follows
\begin{equation*}
\begin{split}
   \int_{\Omega}&
 PU_{x_{\varepsilon,l},\lambda_{\varepsilon,l}}
 \frac{\partial PU^{s}_{x_{\varepsilon,j},\lambda_{\varepsilon,j}}}{\partial {x^i}}\\ =&
 O\Big( \int_{\Omega}
 \frac{\lambda_{\e,l}^{\frac{N-2}{2}}}{\big(1+\lambda_{\e,l}^2|x-x_{\e,l}|^2\big)^{\frac{N-2}{2}}}
 \cdot \frac{\lambda_{\e,j}^{\frac{(N-2)s}{2}+2}}{\big(1+\lambda_{\e,j}^2|x-x_{\e,j}|^2\big)^{\frac{(N-2)s}{2}+1}}
 |x-x_{\e,j}|
   \Big)\\ =&
 O\Big( \int_{B_d(x_{\e,j})}\frac{1}{ \lambda_{\e,l}^{\frac{N-2}{2}}}
 \frac{\lambda_{\e,j}^{\frac{(N-2)s}{2}+2}|x-x_{\e,j}|}{\big(1+\lambda_{\e,j}^2
 |x-x_{\e,j}|^2\big)^{\frac{(N-2)s}{2}+1}}
\Big)
+O\Big(  \frac{1}{ \lambda_{\e,l}^{\frac{N-2}{2}}\lambda_{\e,j}^{\frac{(N-2)s}{2}}} \Big)\\ &
+O\Big(
  \frac{1}{ \lambda_{\e,j}^{\frac{(N-2)s}{2}}} \int_{B_d(x_{\e,l})}
 \frac{\lambda_{\e,l}^{\frac{N-2}{2}}}{\big(1+\lambda_{\e,l}^2|x-x_{\e,l}|^2\big)^{\frac{N-2}{2}}}
\Big)
\\
  =&   O\Big(   \frac{ \lambda_{\e,j}}{ \lambda_{\e,l}^{\frac{N-2}{2}   }     }   \eta(\lambda_{\e,j}) + \frac{1}{ \lambda_{\e,l}^{\frac{N-2}{2}}\lambda_{\e,j}^{\frac{(N-2)s}{2}}}  \Big)        = O\Big(\sum^m_{l=1}\frac{1}{\lambda_{\varepsilon,l}^{N-3}} \Big),
\end{split}
\end{equation*}
where $\eta(\lambda)$ is the function in \eqref{ll1}.
\end{proof}

\vskip 0.2cm
\subsection{Existence}~

\vskip 0.1cm
Now we recall  \begin{equation*}
 \begin{split}
 {\mathbb D}_{\varepsilon}=\Big\{(x_\varepsilon,\lambda_\varepsilon)| &~ x_\varepsilon=(x_{\varepsilon,1},\cdots,x_{\varepsilon,m}), ~~\lambda_\varepsilon=
 (\lambda_{\varepsilon,1},\cdots,\lambda_{\varepsilon,m}),  \\& ~|x_{\varepsilon,j}-a_j|=o(1),~ \frac{\lambda_{\varepsilon,j}}{\lambda_{\varepsilon,l}}\leq C~\mbox{and}~\lambda_{\varepsilon,j}\rightarrow +\infty, ~j,l=1,\cdots,m\Big\},
 \end{split}\end{equation*}
 and  ${\mathbb E}_{x_\varepsilon,\lambda_\varepsilon}:=\displaystyle\bigcap^m_{j=1}{\mathbb E}_{x_{\varepsilon,j},\lambda_{\varepsilon,j}}$ with
 \begin{equation*}
\begin{split}
{\mathbb E}_{{x_{\varepsilon,j}},{\lambda_{\varepsilon,j}}}=\left\{v\in H^1_0(\Omega)\big|~~ \Big\langle \frac{\partial PU_{{x_{\varepsilon,j }},{\lambda_{\varepsilon,j}} }}{\partial {\lambda} },v \Big\rangle=\Big\langle \frac{\partial PU_{{x_{\varepsilon,j}}, {\lambda_{\varepsilon,j}} }}{\partial  {x^i }},v\Big\rangle=0,~\mbox{for}~i=1,\cdots,N\right\}.
\end{split}
\end{equation*}
Now for $(x_\varepsilon,\lambda_\varepsilon)\in {\mathbb D}_{\varepsilon}$,  we consider the following equation on ${\mathbb E}_{x_\varepsilon,\lambda_\varepsilon}$:
 \begin{equation}\label{aa2-16-1}
{\bf Q}_\varepsilon v_\varepsilon={\bf f}_\varepsilon+{\bf R}_\varepsilon(v_\varepsilon),
 \end{equation}
where ${\bf Q}_\varepsilon,   {\bf f}_\varepsilon,  {\bf R}_\varepsilon$ are defined in \eqref{Qe}, \eqref{fe}, \eqref{Re}.

\begin{Prop}\label{Prop-Luo1}
For $(x_\varepsilon,\lambda_\varepsilon)\in {\mathbb D}_{\varepsilon}$ and
$\varepsilon\in (0,\varepsilon_0)$ with $\varepsilon_0$ a fixed small constant, there exists $v_\varepsilon(x_\varepsilon,\lambda_\varepsilon)\in {\mathbb E}_{x_\varepsilon,\lambda_\varepsilon}$ satisfying \eqref{aa2-16-1}.
Furthermore $v_\varepsilon$ satisfies
 \begin{equation}\label{a-4-18-12}
 \|v_{\varepsilon}\|=O\Big(\varepsilon \eta_1(\widetilde{\lambda}_\varepsilon)+ \frac{1}{\widetilde{\lambda}^2_\varepsilon}+\sum^m_{j=1} |x_{\e,j}-a_j|^2  \Big),
\end{equation}
where $\widetilde{\lambda}_\varepsilon=:\min \big\{\lambda_{1,\varepsilon},\cdots,\lambda_{m,\varepsilon}\big\}$.
\end{Prop}
\begin{proof}
First, from Proposition \ref{Prop-Luo2}, ${\bf Q}_\e$ is invertible on ${\mathbb E}_{x_\varepsilon,\lambda_\varepsilon}$ and $\|{\bf Q}_\e^{-1}\|\leq C$ for some $C>0$ independent of $x_\varepsilon,\lambda_\varepsilon,\varepsilon$.
   Furthermore,   by using  Riesz  representation theorem, this equation can be  rewritten   in  $    {\mathbb E}_{x_\varepsilon,\lambda_\varepsilon}$   as the operational form
\[       v_\e =     {\bf Q}_\e^{-1}  \Big[   {\bf f}_\varepsilon+{\bf R}_\varepsilon(v_\varepsilon)  \Big]           : = {\bf A }   (  v_\e)        \quad \text{for} ~   v_\e  \in     {\mathbb E}_{x_\varepsilon,\lambda_\varepsilon}.  \]
For small fixed $\tau>0$, we define
 \[     \mathcal{S} =  \Bigg\{    v_\e  \in     {\mathbb E}_{x_\varepsilon,\lambda_\varepsilon}\Big|    ~     \|v_{\varepsilon}\|\leq \varepsilon^{1-\tau} \eta_1(\widetilde{\lambda}_\varepsilon)+ \frac{1}{\widetilde{\lambda}^{2-\tau}_\varepsilon}+\sum^m_{j=1} |x_{\e,j}-a_j|^{2-\tau}        \Bigg\}.
 \]
Also we have the following estimate on  ${\bf R}_\varepsilon(v_\var)$:
\begin{equation} \label{4.23}
\int_{\Omega} {\bf R}_\varepsilon(v_\var)v_\var= O\Big(\|v_\var\|^{\min\{3,2^*\}}+\|v_\var\|^{s+1} \Big)
+O\left(  \e \int_\Omega\Big(\sum^m_{j=1}PU_{x_{\varepsilon,j},\lambda_{\varepsilon,j}}
 \Big)^{s}v_\e {\mathrm d}x \right).
\end{equation}
Combining   \eqref{4.23} and the results of Lemma \ref{lem1},  we can prove the operator ${\bf A } $  is  contraction mapping from  $\mathcal{S}$ to $\mathcal{S}$.
Then by using  fixed  point Theorem,  we can show that there exists $v_\varepsilon$ satisfying \eqref{2-16-1} and \eqref{a-4-18-12}.
\end{proof}
Let $u_\varepsilon=\displaystyle\sum^m_{j=1} Q(a_j) ^{ - \frac {N-2} {4 } }  PU_{x_{\varepsilon,j},
\lambda_{\varepsilon,j}}+  v_{\varepsilon}$, then \eqref{2-16-1} gives us
\begin{equation*}
-\Delta u_\varepsilon= Q(x) u_\varepsilon^{\frac{N+2}{N-2}}+\varepsilon u^{s}_{\varepsilon}+\sum^m_{j=1}\sum^N_{i=0}c_{\varepsilon,i,j}
\varphi_{ij}(x),
\end{equation*}
with $\varphi_{0j}(x)=\frac{\partial PU_{x_{\varepsilon,j},\lambda_{\varepsilon,j}}}{\partial \lambda}$ and $\varphi_{ij}(x)=\frac{\partial PU_{x_{\varepsilon,j},\lambda_{\varepsilon,j}}}{\partial  x^i}$ for $j=1,\cdots,m$ and $i=1,\cdots,N$.

  \begin{Prop}\label{Prop1.2} If $(x_\var,\lambda_\var)\in {\mathbb D}_{\varepsilon}$, then  \eqref{a22-1-8} is equivalent to

  \begin{equation}\label{2020-01-02-11}
 \frac{1}{\lambda^3_{\varepsilon,j}}=o\Big(
\frac{ |x_{\varepsilon,j}-a_j|^2 + \varepsilon \eta_1(\widetilde{\lambda}_\varepsilon)  }{\widetilde{\lambda}_\e }
   \Big) - \begin{cases}
 \Big(\frac{N(4-(N-2)(s-1))B}{2A(N-2)\Delta Q(a_j)(s+1)}+o(1)\Big)
\frac{\varepsilon}{\lambda_{\varepsilon,j}^{3-\frac{N-2}{2}(s-1)}},&~\mbox{for}~N+s\neq 5,\\
 \Big(\frac{2\omega_4}{A\Delta Q(a_j)}+o(1)\Big)\frac{\varepsilon \log \lambda_{\varepsilon,j}}{\lambda_{\varepsilon,j}^{3}},&~\mbox{for}~N+s=5.
\end{cases}
  \end{equation}
   and  \eqref{a2-13-1}  is equivalent to
  \begin{equation}\label{2020-01-02-12}
 \sum^m_{l=1}\frac{\partial^2Q(a_j)}{\partial x^i\partial x^l}\Big({x_{\e,j}^l} -{a_j^l} \Big)=
 O\Big(\frac{1}{\widetilde{\lambda}_\e^{N-3}} \Big)+o\Big( \sum^m_{q=1}|x_{\e,q}-a_q|\Big).
  \end{equation}
  \end{Prop}
 \begin{proof}
 \textbf{Step 1.}   Now we compute  \eqref{a22-1-8} term by term.
 \begin{equation}\label{4-26-5}
 \begin{split}
 \int_{\Omega}&\Big(\Delta u_\varepsilon+
 Q(x) u_\varepsilon^{\frac{N+2}{N-2}}  \Big)\frac{\partial PU_{x_{\varepsilon,j},\lambda_{\varepsilon,j}}}{\partial
  \lambda }\\ =&
   \sum^m_{l=1}  Q(a_l) ^{ - \frac {N+2} {4 } }  \int_{\Omega}
\Big(Q(x)-  Q(a_l) \Big) PU_{x_{\varepsilon,l},\lambda_{\varepsilon,l}}^{\frac{N+2}{N-2}} \frac{\partial PU_{x_{\varepsilon,j},\lambda_{\varepsilon,j}}}{\partial
  \lambda }+ \int_{\Omega}
   \Delta v_\e\frac{\partial PU_{x_{\varepsilon,j},\lambda_{\varepsilon,j}}}{\partial
  \lambda }\\ &+
   \int_{\Omega}
Q(x) \Big( u_\varepsilon^{\frac{N+2}{N-2}} -  \sum^m_{l=1}  Q(a_l) ^{ - \frac {N+2} {4 } }  PU_{x_{\varepsilon,l},\lambda_{\varepsilon,l}}^{\frac{N+2}{N-2}}  \Big) \frac{\partial PU_{x_{\varepsilon,j},\lambda_{\varepsilon,j}}}{\partial
  \lambda }.
  \end{split}\end{equation}
Also by orthogonality, we know
\begin{equation}\label{4-26-6}
\begin{split}
\int_{\Omega}
   \Delta v_\e\frac{\partial PU_{x_{\varepsilon,j},\lambda_{\varepsilon,j}}}{\partial
  \lambda }
=-\int_{\Omega}
  \nabla v_\e\frac{\partial \nabla PU_{x_{\varepsilon,j},\lambda_{\varepsilon,j}}}{\partial
  \lambda }=0.
\end{split}
\end{equation}
And from \eqref{3-14-11}, we have
 \begin{equation} \label{4-24-22}
 \begin{split}
  u_\varepsilon&^{\frac{N+2}{N-2}}= Q(a_l) ^{ - \frac {N+2} {4 } }  PU_{x_{\varepsilon,j},\lambda_{\varepsilon,j}}^{\frac{N+2}{N-2}}
  +   \frac{N+2}{N-2} Q(a_j) ^{-1}  PU_{x_{\varepsilon,j},\lambda_{\varepsilon,j}}^{\frac{4}{N-2}} \big(\sum_{l\neq j} Q(a_l) ^{ - \frac {N-2} {4 } }  PU_{x_{\varepsilon,l},\lambda_{\varepsilon,l}}
  +v_\e \big) \\ &
  + \begin{cases}
 O\Big(\big(\displaystyle\sum_{l\neq j} PU_{x_{\varepsilon,l},\lambda_{\varepsilon,l}}+v_\e \big)^{\frac{N+2}{N-2}} \Big),~&\mbox{for}~N\geq 6, \\[6mm]
 O\Big(PU_{x_{\varepsilon,j},\lambda_{\varepsilon,j}}^{\frac{6-N}{N-2}} \big(\displaystyle\sum_{l\neq j}PU_{x_{\varepsilon,l},\lambda_{\varepsilon,l}}+v_\e \big)^{2} +\big(\displaystyle\sum_{l\neq j} PU_{x_{\varepsilon,l},\lambda_{\varepsilon,l}}+v_\e \big)^{\frac{N+2}{N-2}} \Big),~&\mbox{for}~N<6.
  \end{cases}
  \end{split}\end{equation}
Also by orthogonality, we get
 \begin{equation}\label{4-26-1}
\begin{split}
   \int_{\Omega}&
Q(x)  v_\e
 \frac{\partial PU^{\frac{N+2}{N-2}}_{x_{\varepsilon,j},\lambda_{\varepsilon,j}}}{\partial
  \lambda }\\=&\frac {N+2}{N-2}
 \int_{\Omega}\big(Q(x)-Q(a_j)\big)
  v_\e  PU^{\frac{ 4}{N-2}}_{x_{\varepsilon,j},\lambda_{\varepsilon,j} }
 \frac{\partial PU_{x_{\varepsilon,j},\lambda_{\varepsilon,j}}}{\partial
  \lambda} \\ =& O\left( \Big(\int_{\Omega}  \big(Q(x)-Q(a_j)\big)^{\frac{2N}{N+2} }PU^{\frac{4}{N-2}\frac{2N}{N+2}}_{x_{\varepsilon,j},\lambda_{\varepsilon,j}}
 \big( \frac{\partial PU_{x_{\varepsilon,j},\lambda_{\varepsilon,j}}}{\partial
  \lambda}\big)^\frac{2N}{N+2}\Big)^{\frac {N+2}{2N}}\| v_\e\|\right)\\=
  &O\Big(\frac 1 {\lambda_{\varepsilon,j}}\Big(\int_{\Omega}|x-x_{\varepsilon,j}|^{\frac{4N}{N+2}}
   PU^{\frac{ 2N}{N-2}} _{x_{\varepsilon,j},\lambda_{\varepsilon,j}}\Big)^{\frac{N+2}{2N}}\|v_\e\|\Big)   \\=&
 o\Big(\frac{1}{\widetilde{\lambda}_\e^{3}}
   \Big)+o\Big(
\frac{ |x_{\varepsilon,j}-a_j|^2  + \varepsilon \eta_1(\widetilde{\lambda}_\varepsilon)   }{\widetilde{\lambda}_\e }
   \Big).
\end{split}
\end{equation}
So using \eqref{4-26-2}, \eqref{4-24-22} and \eqref{4-26-1}, we get
  \begin{equation}\label{4-26-7}
 \begin{split}
   \int_{\Omega}
Q(x) \left( u_\varepsilon^{\frac{N+2}{N-2}} -  \sum^m_{l=1} Q(a_l) ^{ - \frac {N+2} {4 } }  PU_{x_{\varepsilon,l},\lambda_{\varepsilon,l}}^{\frac{N+2}{N-2}}  \right) \frac{\partial PU_{x_{\varepsilon,j},\lambda_{\varepsilon,j}}}{\partial
  \lambda }
   =o\Big(\frac{1}{\widetilde{\lambda}_\e^{3}}
   \Big)+o\Big(
\frac{ |x_{\varepsilon,j}-a_j|^2 +  \e  \eta_1(\widetilde{\lambda}_\varepsilon) }{\widetilde{\lambda}_\e }
   \Big).
  \end{split}\end{equation}
  Hence from \eqref{3-13-21}, \eqref{4-26-5}, \eqref{4-26-6} and  \eqref{4-26-7}, we find
   \begin{equation}\label{3-14-26}
 \begin{split}
 \int_{\Omega}&\Big(\Delta u_\varepsilon+
 Q(x) u_\varepsilon^{\frac{N+2}{N-2}}  \Big)\frac{\partial PU_{x_{\varepsilon,j},\lambda_{\varepsilon,j}}}{\partial
  \lambda } \\ =& - \frac{\Delta Q(a_j)(N-2)A}{N\lambda^3_{\varepsilon,j}}  \big( 1+ o(1) \big)+o\Big(\frac{1}{\widetilde{\lambda}_\e^{3}}
   \Big)+o\Big(
\frac{ |x_{\varepsilon,j}-a_j|^2   +  \e \eta_1(\widetilde{\lambda}_\varepsilon) }{\widetilde{\lambda}_\e }
   \Big).
  \end{split}\end{equation}
On the other hand,
 \begin{equation}\label{4-26-11}
 \begin{split}
 \int_{\Omega}&u^{s}_{\varepsilon}\frac{\partial PU_{x_{\varepsilon,j},\lambda_{\varepsilon,j}}}{\partial \lambda}=
\int_{\Omega}
 PU_{x_{\varepsilon,j},\lambda_{\varepsilon,j}}^{s} \frac{\partial PU_{x_{\varepsilon,j},\lambda_{\varepsilon,j}}}{\partial  \lambda}+
   \int_{\Omega}
 \Big( u_\varepsilon^{s} - PU_{x_{\varepsilon,j},\lambda_{\varepsilon,j}}^{s}  \Big) \frac{\partial PU_{x_{\varepsilon,j},\lambda_{\varepsilon,j}}}{\partial  \lambda}.
  \end{split}\end{equation}
  And we know
 \begin{equation}\label{4-26-12}
 \begin{split}
  u_\varepsilon^{s}-PU_{x_{\varepsilon,j},\lambda_{\varepsilon,j}}^{s}
=O\left(PU_{x_{\varepsilon,j},\lambda_{\varepsilon,j}}^{s-1} \Big(\sum_{l\neq j} PU_{x_{\varepsilon,l},\lambda_{\varepsilon,l}}
  +v_\e \Big)+
 \Big(\displaystyle\sum_{l\neq j} PU_{x_{\varepsilon,l},\lambda_{\varepsilon,l}}+v_\e \Big)^{s}\right).
  \end{split}\end{equation}
  Also
 we find
  \begin{equation}\label{4-26-13}
 \begin{split}
 \int_{\Omega}\Big|v_{\varepsilon}\frac{\partial PU^s_{x_{\varepsilon,j},\lambda_{\varepsilon,j}}}{\partial \lambda}\Big|
 + \int_{\Omega}\Big|v^s_{\varepsilon}\frac{\partial PU_{x_{\varepsilon,j},\lambda_{\varepsilon,j}}}{\partial \lambda}\Big|= O\Big(\frac{\eta_1({\lambda}_{\varepsilon,j})}{\lambda_{\e,j}} \Big) \|v_\e\|.
  \end{split}\end{equation}
 Then from \eqref{3-13-51}, \eqref{a4-26-2}, \eqref{4-26-11}, \eqref{4-26-12} and \eqref{4-26-13}, we get
\begin{equation}\label{3-14-27}
\begin{split}
     \int_{\Omega}
      u^{s}_{\varepsilon}\frac{\partial PU_{x_{\varepsilon,j},\lambda_{\varepsilon,j}}}{\partial
      \lambda}{\mathrm d}x
             = \begin{cases}
- \frac{4-(N-2)(s-1)}{2(s+1)}\Big(B+o(1)\Big)
\frac{1}{\lambda_{\varepsilon,j}^{3-\frac{N-2}{2}(s-1)}},&~\mbox{for}~N+s\neq 5,\\
-\omega_4\frac{\log \lambda_{\varepsilon,j}}{\lambda_{\varepsilon,j}^{3}}
+O\Big(\frac{1}{\lambda_{\varepsilon,j}^{3}}\Big),&~\mbox{for}~N+s=5.
\end{cases}  \end{split}
\end{equation}
Hence \eqref{2020-01-02-11} follows by \eqref{3-14-26} and \eqref{3-14-27}.

\vskip 0.2cm

\noindent \textbf{Step 2.}
 Now we compute  \eqref{a2-13-1} term by term.
 \begin{equation}\label{4-26-21}
 \begin{split}
 \int_{\Omega}&\Big(\Delta u_\varepsilon+
 Q(x) u_\varepsilon^{\frac{N+2}{N-2}}   \Big) \frac{\partial PU_{x_{\varepsilon,j},\lambda_{\varepsilon,j}}}{\partial {x^i}}\\ =&
   \sum^m_{l=1}  Q(a_l) ^{ - \frac {N+2} {4 } }  \int_{\Omega}
\Big(Q(x)-Q(a_l)  \Big)  PU_{x_{\varepsilon,l},\lambda_{\varepsilon,l}}^{\frac{N+2}{N-2}} \frac{\partial PU_{x_{\varepsilon,j},\lambda_{\varepsilon,j}}}{\partial {x^i}}+ \int_{\Omega}
   \Delta v_\e\frac{\partial PU_{x_{\varepsilon,j},\lambda_{\varepsilon,j}}}{\partial {x^i}}\\ &+
   \int_{\Omega}
Q(x) \Big( u_\varepsilon^{\frac{N+2}{N-2}} -  \sum^m_{l=1} Q(a_l) ^{ - \frac {N+2} {4 } }   PU_{x_{\varepsilon,l},\lambda_{\varepsilon,l}}^{\frac{N+2}{N-2}}   \Big)  \frac{\partial PU_{x_{\varepsilon,j},\lambda_{\varepsilon,j}}}{\partial {x^i}}.
  \end{split}\end{equation}
Also by orthogonality, it holds
\begin{equation}\label{4-26-22}
\begin{split}
\int_{\Omega}
   \Delta v_\e\frac{\partial PU_{x_{\varepsilon,j},\lambda_{\varepsilon,j}}}{\partial {x^i}}
=-\int_{\Omega}
  \nabla v_\e\frac{\partial \nabla PU_{x_{\varepsilon,j},\lambda_{\varepsilon,j}}}{\partial {x^i}}=0.
\end{split}
\end{equation}
From \eqref{aa4-26-2}, \eqref{a-4-18-12} and \eqref{4-24-22}, we know that
\begin{equation}\label{4-26-23}
\begin{split}
 \int_{\Omega}
 & Q(x) \Big( u_\varepsilon^{\frac{N+2}{N-2}} -  \sum^m_{l=1} Q(a_l) ^{ - \frac {N+2} {4 } }   PU_{x_{\varepsilon,l},\lambda_{\varepsilon,l}}^{\frac{N+2}{N-2}}   \Big)  \frac{\partial PU_{x_{\varepsilon,j},\lambda_{\varepsilon,j}}}{\partial {x^i}}   \\   &  =
\frac{N+2}{N-2}  Q(a_j) ^{-1}   \int_{\Omega}
Q(x)       PU_{x_{\varepsilon,j},\lambda_{\varepsilon,j}}^{\frac{4}{N-2}} \Big(\sum_{l\neq j} Q(a_l) ^{ - \frac {N-2} {4 } }  PU_{x_{\varepsilon,l},\lambda_{\varepsilon,l}}
  +v_\e \Big)  \frac{\partial PU_{x_{\varepsilon,j},\lambda_{\varepsilon,j}}}{\partial {x^i}}     \\  &
  =O\Big( \frac{1}{\widetilde{\lambda}_\e^{N-3}} \Big)+
 o\Big( \sum^m_{q=1}|x_{\e,q}-a_q|\Big). \end{split}
\end{equation}
  Hence by \eqref{3-13-31}, \eqref{4-26-21}, \eqref{4-26-22} and \eqref{4-26-23}, we find
   \begin{equation}\label{3-14-28}
 \begin{split}
 \int_{\Omega}&\Big(\Delta u_\varepsilon+
 Q(x) u_\varepsilon^{\frac{N+2}{N-2}}  \Big) \frac{\partial PU_{x_{\varepsilon,j},\lambda_{\varepsilon,j}}}{\partial {x^i}}\\ =&
\Big(2\int_{\R^N} \frac{|x|^2}{(1+|x|^2)^{N+1}}{\mathrm d}x +o(1)\Big)
    \sum^m_{l=1}\frac{\partial^2Q(a_j)}{\partial x^i\partial x^l}\Big({x_{\e,j}^l} -{a_j^l} \Big)  +
 O\Big( \frac{1}{\widetilde{\lambda}_\e^{N-3}} \Big).
  \end{split}\end{equation}
On the other hand,
 \begin{equation}\label{4-26-31}
 \begin{split}
 \int_{\Omega}&u^{s}_{\varepsilon}\frac{\partial PU_{x_{\varepsilon,j},\lambda_{\varepsilon,j}}}{\partial {x^i}}=\int_{\Omega}
 PU_{x_{\varepsilon,j},\lambda_{\varepsilon,j}}^{s} \frac{\partial PU_{x_{\varepsilon,j},\lambda_{\varepsilon,j}}}{\partial {x^i}}+
   \int_{\Omega}
 \Big( u_\varepsilon^{s} -  PU_{x_{\varepsilon,j},\lambda_{\varepsilon,j}}^{s}  \Big) \frac{\partial PU_{x_{\varepsilon,j},\lambda_{\varepsilon,j}}}{\partial {x^i}}.
  \end{split}\end{equation}
Also
 we find
  \begin{equation}\label{4-26-32}
 \begin{split}
 \int_{\Omega}\Big|v_{\varepsilon}\frac{\partial PU^s_{x_{\varepsilon,j},\lambda_{\varepsilon,j}}}{\partial {x^i}}\Big|
 + \int_{\Omega}\Big|v^s_{\varepsilon}\frac{\partial PU_{x_{\varepsilon,j},\lambda_{\varepsilon,j}}}{\partial {x^i}}\Big|= O\Big(\lambda_{\e,j}\eta_1({\lambda}_{\varepsilon,j})  \Big) \|v_\e\|.
  \end{split}\end{equation}
 Then from \eqref{3-13-41}, \eqref{ab4-26-2}, \eqref{4-26-31} and \eqref{4-26-32}, we get
\begin{equation}\label{3-14-29}
\begin{split}
     \int_{\Omega}
      u^{s}_{\varepsilon}\frac{\partial PU_{x_{\varepsilon,j},\lambda_{\varepsilon,j}}}{\partial
      \lambda}{\mathrm d}x
       =O\Big( \frac{1}{\widetilde{\lambda}_\e^{N-3}} \Big).
\end{split}
\end{equation}
Hence \eqref{2020-01-02-12} can follows by \eqref{3-14-28} and \eqref{3-14-29}.

  \end{proof}
\begin{proof}[\textbf{Proof of Theorem \ref{th1.2}}]
 For $j=1,\cdots,m$, let
 \begin{equation}\label{3-11-1}
 \begin{split}
 \lambda_{\var,j}\in \begin{cases}
 \Big[c_1\varepsilon^{-\big(\frac{N-2}{2}(s-1)\big)^{-1}},c_2\varepsilon^{-\big(\frac{N-2}{2}(s-1)\big)^{-1}}\Big],
 ~&~\mbox{if}~N\geq 4 ~\mbox{and}~s>1,\\[3mm]
 \Big[e^{\frac{c_1}{\varepsilon}},e^{\frac{c_2}{\varepsilon}}\Big],
 ~&~\mbox{if}~N=4 ~\mbox{and}~s=1,
 \end{cases}
 \end{split}\end{equation}
 and
  \begin{equation}\label{a3-11-1}
 {x_{\varepsilon,j}^i} \in \begin{cases}
 \Big[a_{j,i}-\varepsilon^{\frac{1}{2(s-1)}}, a_{j,i}+\varepsilon^{\frac{1}{2(s-1)}}\Big],&~\mbox{if}~N\geq 4~\mbox{and}~s>1,\\[3mm]
  \Big[a_{j,i}-e^{-\frac{c_1}{2\varepsilon}}, a_{j,i}+e^{-\frac{c_1}{2\varepsilon}}\Big],
 ~&~\mbox{if}~N=4 ~\mbox{and}~s=1,
 \end{cases} \end{equation}
 where $c_1,c_2$ are fixed positive constants with $c_1$  small and  $c_2$ large.
Hence
let $$M:=\Big\{(x_{\varepsilon},\lambda_\varepsilon), (x_{\varepsilon},\lambda_\varepsilon) ~\mbox{satisfying}~\eqref{3-11-1}~\mbox{and}~\eqref{a3-11-1}\Big\},$$
and some vectors $y=(y_1,\cdots,y_m)$, $\lambda=(\lambda_1,\cdots,
\lambda_m)$. Also for $i=1,\cdots,N,~j=1,\cdots,m$, we define
  \begin{equation*}
 F_{m(i-1)+j}(y,\lambda) =\sum^m_{l=1}\frac{\partial^2Q(a_j)}{\partial x^i\partial x^l}\Big({y_{j}^l} -{a_j^l} \Big)+
 O\Big(\frac{1}{|\lambda|^{N-3}} \Big)+o\Big( \sum^m_{q=1}|y_{q}-a_q|\Big),
  \end{equation*}
  and
   \begin{equation*}
   \begin{split}
 F_{mN+j}(y,\lambda)=&\frac{1}{\lambda^3_{j}}+o\Big(
\frac{ |y_{j}-a_j|^2 + \varepsilon \eta_1(|\lambda|)}{|\lambda|}
   \Big) \\&+\begin{cases}
 \Big(\frac{N(4-(N-2)(s-1))B}{2A(N-2)\Delta Q(a_j)(s+1)}+o(1)\Big)
\frac{\varepsilon}{\lambda_{j}^{3-\frac{N-2}{2}(s-1)}},&~\mbox{for}~N+s\neq 5,\\
 \Big(\frac{2\omega_4}{A\Delta Q(a_j)}+o(1)\Big)\frac{\varepsilon \log \lambda_{j}}{\lambda_{j}^{3}},&~\mbox{for}~N+s=5.
\end{cases}\end{split}
  \end{equation*}
  By orthogonal transform, we can view the matrix
  $\Big(\frac{\partial^2Q(a_j)}{\partial x^i\partial x^l}\Big)_{1\leq i,l\leq N}$ as a diagonal matrix
  $diag(\mu_1,\cdots,\mu_N)$.
  Then for $(y,\lambda)\in M$, $N\geq 4$ and $s>1$,  it holds
  \begin{equation*}
\Big(F_{m(i-1)+j}(y,\lambda)|_{y^i_j=a_{j,i}-\varepsilon^{\frac{1}{2(s-1)}}} \Big)\cdot \Big(F_{m(i-1)+j}(y,\lambda)|_{y^i_j=a_{j,i}+\varepsilon^{\frac{1}{2(s-1)}}}\Big)=
-\Big(\mu_i^2+o(1)\Big)\varepsilon^{\frac{1}{s-1}}<0,
  \end{equation*}
  \begin{equation*}
 F_{mN+j}(y,\lambda)|_{\lambda_j=c_1\varepsilon^{-\big(\frac{N-2}{2}(s-1)\big)^{-1}}}
 =
  \frac{1}{\lambda_{j}^{3}}\Big(1+\frac{N(4-(N-2)(s-1))B}{2A(N-2)\Delta Q(a_j)(s+1)}\big(\frac{1}{c_1}\big)^{\frac{N-2}{2}(s-1)}+o(1)\Big)
< 0,
  \end{equation*}
  and
  \begin{equation*}
 F_{mN+j}(y,\lambda)|_{\lambda_j=c_2\varepsilon^{-\big(\frac{N-2}{2}(s-1)\big)^{-1}}}
 = \frac{1}{\lambda_{j}^{3}}\Big(1+\frac{N(4-(N-2)(s-1))B}{2A(N-2)\Delta Q(a_j)(s+1)}\big(\frac{1}{c_2}\big)^{\frac{N-2}{2}(s-1)}+o(1)\Big)>
 0.
  \end{equation*}
 Then by well known Poincar\'e-Miranda theorem (see Lemma \ref{lem-B-1} in the Appendix),  \eqref{2020-01-02-11} and \eqref{2020-01-02-12}  have  a solution
  $(x_{\varepsilon},\lambda_\varepsilon)\in M$ when $N\geq 4$ and $s>1$. Similarly, there exists
  $(x_{\varepsilon},\lambda_\varepsilon)\in M$  which solves \eqref{2020-01-02-11} and \eqref{2020-01-02-12} when $N=4$ and $s=1$.

   On the other hand, by finite dimensional reduction (Proposition \ref{Prop-Luo1}), we know that there exists $u_\varepsilon$ satisfying
    \begin{equation*}
    -\Delta u_\varepsilon= Q(x) u_\varepsilon^{\frac{N+2}{N-2}}+\varepsilon
    u^{s}_{\varepsilon}+\sum^m_{j=1}\sum^N_{i=0}c_{\varepsilon,i,j}
    \varphi_{ij}(x).
    \end{equation*}
Also from
\textbf{Claim 1} in page 3 and Proposition \ref{Prop1.2}, we find
$$c_{\varepsilon,0,j}=c_{\varepsilon,i,j}=0~\mbox{for}~i=1,\cdots,N~ \mbox{and}~j=1,\cdots,m.$$
Hence
 $u_\varepsilon$ is a concentrated solution of \eqref{1.1}, which satisfies \eqref{4-6-1}. \end{proof}

\section{Local uniqueness (Proof of Theorem \ref{th1.3})
}
\setcounter{equation}{0}

Let $u^{(1)}_{\varepsilon}(x)$, $u^{(2)}_{\varepsilon}(x)$ be two different solutions of $\eqref{1.1}$ satisfying \eqref{4-6-1}. Under Condition (Q), $N\geq 4$ and $ s\in[1,\frac{N+2}{N-2})$, we find from Theorem \ref{prop1} that
$u^{(l)}_{\varepsilon}(x)$ can be written as
\begin{equation*}
u^{(l)}_{\varepsilon}=\sum^m_{j=1}\big(Q(a_j)\big)^{-\frac{N-2}{4}} PU_{x^{(l)}_{\varepsilon,j}, \lambda^{(l)}_{\varepsilon,j}}+w^{(l)}_{\varepsilon},
\end{equation*}
satisfying, for $j=1,\cdots,m$, $l=1,2$,
 $\lambda^{(l)}_{\varepsilon,j}=
\big(u_\varepsilon(x^{(l)}_{\varepsilon,j})\big)^{\frac{2}{N-2}}$,
\begin{equation*}
 x^{(l)}_{\varepsilon,j}\rightarrow a_j, ~ \lambda^{(l)}_{\varepsilon,j} \rightarrow +\infty,~ \|w^{(l)}_{\varepsilon}\|=o(1)~\mbox{and}~w^{(l)}_\varepsilon\in \bigcap^m_{j=1}E_{x^{(l)}_{\varepsilon,j},\lambda^{(l)}_{\varepsilon,j}}.
\end{equation*}
Moreover  we define  $\bar{\lambda}_\varepsilon:=
\min
\Big\{\lambda^{(1)}_{1,\varepsilon},\cdots,
\lambda^{(1)}_{m,\varepsilon},\lambda^{(2)}_{1,\varepsilon},\cdots,
\lambda^{(2)}_{m,\varepsilon}\Big\}$, and from \eqref{lp111}, \eqref{2020-01-02-11} and \eqref{2020-01-02-12}, we find
\begin{equation}\label{4-26-61}
 \big|x^{(l)}_{\varepsilon,j}-a_j\big|=o\big(\frac{1}{\bar{\lambda}_\varepsilon}\big), ~ \lambda^{(l)}_{\varepsilon,j}
 =
 \e^{-\frac{2}{(N-2)(s-1)}}\Big(C_j+o(1)   \Big), ~~\mbox{for}~N\geq 5~\mbox{and}~s>1,
\end{equation}
with some positive constants $C_j$.
 Now we set
\begin{equation}\label{3.1}
\xi_{\varepsilon}(x)=\frac{u_{\varepsilon}^{(1)}(x)-u_{\varepsilon}^{(2)}(x)}
{\|u_{\varepsilon}^{(1)}-u_{\varepsilon}^{(2)}\|_{L^{\infty}(\Omega)}},
\end{equation}
then $\xi_{\varepsilon}(x)$ satisfies $\|\xi_{\varepsilon}\|_{L^{\infty}(\Omega)}=1$ and
\begin{equation}\label{3.2}
-\Delta \xi_{\varepsilon}(x)=C_{\varepsilon}(x)\xi_{\varepsilon}(x),
\end{equation}
where
\begin{equation*}
C_{\varepsilon}(x)=\Big(2^*-1   \Big)Q(x)\int_{0}^1
\Big(tu_{\varepsilon}^{(1)}(x)+(1-t)u_{\varepsilon}^{(2)}(x)   \Big)
^{\frac{4}{N-2}}{\mathrm d}  t+\e s \int_{0}^1
\Big(tu_{\varepsilon}^{(1)}(x)+(1-t)u_{\varepsilon}^{(2)}(x)   \Big)
^{s-1}{\mathrm d}  t.
\end{equation*}

Next, we give  some estimates on $\xi_\e$.
\begin{Prop}
For $N\geq 5$ and $\xi_{\varepsilon}(x)$ defined by \eqref{3.1}, we have
\begin{equation}\label{3-3}
\xi_{\varepsilon}(x) =O\Big(\frac{\log \bar{\lambda}_\varepsilon}{\bar{\lambda}^{N-2}_\varepsilon}  \Big),~\mbox{in}~ C^1\Big(\Omega\backslash\bigcup_{j=1}^mB_{d}(x_{\varepsilon,j}^{(1)})\Big),
\end{equation}
where $d>0$ is any small fixed constant.
\end{Prop}
\begin{proof}
By the potential theory and  \eqref{3.2},  we have
\begin{equation}\label{cc6}
\begin{split}
 {\xi}_\varepsilon(x)=& \int_{\Omega}G(y,x)
 C_{\varepsilon}(y) \xi_{\varepsilon}(y){\mathrm d}y\\=&
O\left(\sum^k_{j=1}\sum^2_{l=1}\int_{\Omega}\frac{1}{|x-y|^{N-2}} \Big( U^{\frac{4}{N-2}}_{x^{(l)}_{\varepsilon,j},\lambda^{(l)}_{\varepsilon,j}}
(y)+\varepsilon U^{s-1}_{x^{(l)}_{\varepsilon,j},\lambda^{(l)}_{\varepsilon,j}}
(y)\Big) {\mathrm d}y \right).
\end{split}
\end{equation}
Next, by \eqref{AB.2}, we know
\begin{equation}\label{acc6}
 \int_{\Omega}\frac{1}{|x-y|^{N-2}}  U^{\frac{4}{N-2}}_{x^{(l)}_{\varepsilon,j},\lambda^{(l)}_{\varepsilon,j}}
(y)  {\mathrm d}y =
O\Big( \frac{1}{\big(1+\lambda^{(l)}_{\varepsilon,j}|x-x^{(l)}_{\varepsilon,j}|\big)^{2}}\Big),
\end{equation}
and
\begin{equation}\label{bcc6}
 \int_{\Omega}\frac{1}{|x-y|^{N-2}}  U^{s-1}_{x^{(l)}_{\varepsilon,j},\lambda^{(l)}_{\varepsilon,j}}
(y) {\mathrm d}y =
O\Big( \frac{1}{\big(1+\lambda^{(l)}_{\varepsilon,j}|x-x^{(l)}_{\varepsilon,j}|\big)^{\frac{(N-2)(s-1)}{2}}}\Big).
\end{equation}
Now from \eqref{4-26-61}, \eqref{cc6}, \eqref{acc6} and \eqref{bcc6}, we can compute
\begin{equation}\label{dd}
\begin{split}
 {\xi}_\varepsilon(x)=& \int_{\Omega}G(y,x)
 C_{\varepsilon}(y) \xi_{\varepsilon}(y){\mathrm d}y\\=&
O\left(\sum^m_{j=1}\sum^2_{l=1}
 \frac{1}{\big(1+\lambda^{(l)}_{\varepsilon,j}|x-x^{(l)}_{\varepsilon,j}|\big)^{2}} \right)+
 O\left(\sum^m_{j=1}\sum^2_{l=1}
 \frac{1}{\big(1+\lambda^{(l)}_{\varepsilon,j}|x-x^{(l)}_{\varepsilon,j}|\big)^{(N-2)(s-1)}}\right).
\end{split}
\end{equation}

Next repeating the above process, we know
\begin{equation*}
\begin{split}
{\xi}_\varepsilon(x) = &O\Big( \sum^m_{j=1}\sum^2_{l=1}\int_{\Omega}\frac{1}{|x-y|^{N-2}}
\frac{ C_{\varepsilon}(y) }{\big(1+\lambda^{(l)}_{\varepsilon,j}|x-x^{(l)}_{\varepsilon,j}|\big)^{2}}
 {\mathrm d}y\Big)\\&
+O\Big(  \sum^m_{j=1}\sum^2_{l=1}\int_{\Omega}\frac{1}{|x-y|^{N-2}}
\frac{ C_{\varepsilon}(y)}{\big(1+\lambda^{(l)}_{\varepsilon,j}|x-x^{(l)}_{\varepsilon,j}|\big)^{(N-2)(s-1)}}
{\mathrm d}y\Big)\\=&
O\left(\sum^m_{j=1}\sum^2_{l=1}
 \frac{1}{\big(1+\lambda^{(l)}_{\varepsilon,j}|x-x^{(l)}_{\varepsilon,j}|\big)^{4}} \right)+
 O\left(\sum^m_{j=1}\sum^2_{l=1}
 \frac{1}{\big(1+\lambda^{(l)}_{\varepsilon,j}|x-x^{(l)}_{\varepsilon,j}|\big)^{2(N-2)(s-1)}}\right).
\end{split}
\end{equation*}
Then we can proceed as in the above argument for finite number of times to prove
\begin{equation*}
{\xi}_\varepsilon(x) =
O\Big(\sum^m_{j=1}\sum^2_{l=1}\frac{\log \bar{\lambda}_\varepsilon}{\big(1+\lambda^{(l)}_{\varepsilon,j}
|x-x^{(l)}_{\varepsilon,j}|\big)^{N-2}}\Big).
\end{equation*}
Also  we know  $\frac{\partial
 {\xi}_\varepsilon(x)}{\partial x^i}=  \displaystyle\int_{\Omega}D_{x^i}G(y,x)
 C_{\varepsilon}(y) \xi_{\varepsilon}(y){\mathrm d}y$.
Hence repeating above process, we can complete the proof of \eqref{3-3}.
\end{proof}
\begin{Prop}\label{prop3-2}
Let $\xi_{\varepsilon,j}(x)=\xi_{\varepsilon}(\frac{x}{\lambda_{\varepsilon,j}^{(1)}}+x_{\varepsilon,j}^{(1)})$. Then by taking
a subsequence if necessary, we have
\begin{equation}\label{cc7}
 \xi_{\varepsilon,j}(x)\rightarrow \sum_{i=0}^N a_{j,i}\psi_{i}(x),~\mbox{uniformly in}~C^1\big(B_R(0)\big) ~\mbox{for any}~R>0,
\end{equation}
 where $a_{j,i}$, $i=0,1,\cdots,N$ are some constants  and $$\psi_{0}(x)=\frac{\partial U_{0,\lambda}(x)}{\partial\lambda}\big|_{\lambda=1},~~\psi_{i}(x)=\frac{\partial U_{x,1}(x)}{\partial x^i}\big|_{x=0},~i=1,\cdots,N.
$$
\end{Prop}
\begin{proof}
Since $\xi_{\varepsilon,j}(x)$ is bounded, by the regularity theory in \cite{Gilbarg}, we find
$$\xi_{\varepsilon,j}(x)\in {C^{1,\alpha}\big(B_r(0)\big)}~ \mbox{and}~  \|\xi_{\varepsilon,j}\|_{C^{1,\alpha}\big(B_r(0)\big)}\leq C,$$
for any fixed large $r$ and $\alpha \in (0,1)$ if $\varepsilon$ is small, where the constants $r$ and $C$ are independent of $\varepsilon$ and $j$.
So we may assume that $\xi_{\varepsilon,j}(x)\rightarrow\xi_{j}(x)$ in $C^1\big(B_r(0)\big)$. By direct calculations, we know
\begin{small}
\begin{equation}\label{tian41}
\begin{split}
-\Delta\xi_{\varepsilon,j}(x)=&-\frac{1}{(\lambda^{(1)}_{\varepsilon,j})^{2}}\Delta \xi_{\varepsilon}\big(\frac{x}{\lambda^{(1)}_{\varepsilon,j}}+x_{\varepsilon,j}^{(1)}\big)=
\frac{1}{(\lambda^{(1)}_{\varepsilon,j})^{2}}C_{\varepsilon}(\frac{x}{\lambda^{(1)}_{\varepsilon,j}}
+x_{\varepsilon,j}^{(1)})\xi_{\varepsilon,j}(x).
\end{split}
\end{equation}
\end{small}
Now, we estimate $\frac{1}{(\lambda^{(1)}_{\varepsilon,j})^{2}}C_{\varepsilon}(\frac{x}{\lambda^{(1)}_{\varepsilon,j}}
+x_{\varepsilon,j}^{(1)})$. By \eqref{4-26-61}, we have
\begin{equation*}
\begin{split}
U&_{x_{\varepsilon,j}^{(1)},\lambda_{\varepsilon,j}^{(1)}}(x)
-
U_{x_{\varepsilon,j}^{(2)},\lambda_{\varepsilon,j}^{(2)}}(x)
 \\=&
O\Big(\big|x_{\varepsilon,j}^{(1)}-x_{\varepsilon,j}^{(2)}\big|\cdot \big(\nabla_y U_{y,\lambda_{\varepsilon,j}^{(1)}}(x)|_{y=x_{\varepsilon,j}^{(1)}}\big)+
\big|\lambda_{\varepsilon,j}^{(1)}-\lambda_{\varepsilon,j}^{(2)}\big|\cdot\big( \nabla_\lambda U_{x_{\varepsilon,j}^{(1)},\lambda}(x) |_{\lambda=\lambda_{\varepsilon,j}^{(1)}}\big)
\Big)\\=&
O\Big(\lambda_{\varepsilon,j}^{(1)}\big|x_{\varepsilon,j}^{(1)}-x_{\varepsilon,j}^{(2)}\big| +
(\lambda_{\varepsilon,j}^{(1)})^{-1}|\lambda_{\varepsilon,j}^{(1)}-\lambda_{\varepsilon,j}^{(2)}| \Big) U_{x_{\varepsilon,j}^{(1)},\lambda_{\varepsilon,j}^{(1)}}(x)
=o\big(1\big)U_{x_{\varepsilon,j}^{(1)},\lambda_{\varepsilon,j}^{(1)}}(x),
\end{split}
\end{equation*}
which means
\begin{equation*}
u_{\varepsilon}^{(1)}(x)-u_{\varepsilon}^{(2)}(x)=
o\big(1\big)\Big(\sum_{j=1}^mU_{x_{\varepsilon,j}^{(1)},\lambda_{\varepsilon,j}^{(1)}}(x)\Big)
+O\Big(\sum^2_{l=1}|w_{\varepsilon}^{(l)}(x)|\Big).
\end{equation*}
Then for a small fixed $d>0$, we find
\begin{equation*}
\begin{split}\frac{1}{(\lambda^{(1)}_{\varepsilon,j})^{2}}C_{\varepsilon}
(\frac{x}{\lambda^{(1)}_{\varepsilon,j}}
+x_{\varepsilon,j}^{(1)}) \rightarrow  \frac{N+2}{N-2} U^{\frac{4}{N-2}}_{0,1}(x),~\mbox{for}~\frac{x}{\lambda^{(1)}_{\varepsilon,j}}
+x_{\varepsilon,j}^{(1)}\in B_{d}(0).
\end{split}
\end{equation*}
Letting $\varepsilon\rightarrow 0$ in \eqref{tian41} and using the elliptic regularity theory, we find that $\xi_{j}(x)$ satisfies
\begin{equation*}
-\Delta\xi_{j}(x)=\Big(\frac{N+2}{N-2}\Big)U_{0,1}^{\frac{4}{N-2}}(x)\xi_{j}(x),~~\textrm{in~}\R^N,
\end{equation*}
which gives $
\xi_{j}(x)=\displaystyle\sum_{i=0}^Na_{j,i}\psi_i(x)$.
\end{proof}

\begin{Prop}
For $N\geq 5$, it holds
\begin{equation}\label{luo--13}
a_{j,i}=0,~\mbox{for}~j=1,\cdots,m~\mbox{and}~i=1,\cdots,N,
\end{equation}
where $a_{j,i}$ are the constants in \eqref{cc7}.
\end{Prop}
\begin{proof}
First using  \eqref{3-3} and \eqref{cc7}, we compute that
\begin{equation}\label{dclp-2}
\begin{split}
\mbox{LHS of \eqref{dclp-1}}=\big(Q(a_j)\big)^{-\frac{N+2}{4}}\left(
\sum^N_{l=1}\Big(\int_{\R^N}x^jU_{0,1}^{\frac{N+2}{N-2}}(x)\psi_j(x)   \Big)\frac{\partial^2Q(a_j)}{\partial x^i\partial x_l}{a_j^l}+o(1)\right)\frac{1}{\big(\lambda^{(1)}_{j,\e}\big)^{\frac{N}{2}}}.
\end{split}\end{equation}
And from \eqref{ab4-18-12} and \eqref{3-3}, it holds
\begin{equation}\label{dclp-3}
\begin{split}
\mbox{RHS of \eqref{dclp-1}}= O\Big(\frac{\log \bar \lambda_{\e}}{\bar \lambda_{\e}^{\frac{3(N-2)}{2}}}  \Big).
\end{split}\end{equation}
Hence we find \eqref{luo--13} by \eqref{dclp-2} and \eqref{dclp-3}.
\end{proof}
\begin{Prop}
For $N\geq 5$, it holds
\begin{equation}\label{luo-13}
a_{j,0}=0,~\mbox{for}~j=1,\cdots,m,
\end{equation}
where $a_{j,0}$ are the constants in Proposition \ref{prop3-2}.
\end{Prop}
\begin{proof}
First using  \eqref{3-3} and \eqref{cc7}, we compute that
\begin{equation*}
\begin{split}
\int_{B_d(x^{(1)}_{\e,j})}&   ( x-x_{\varepsilon,j}) \cdot \nabla Q(x)      D_{1,\varepsilon}\xi_\e {\mathrm d}x
\\ =&\big(Q(a_j)\big)^{-\frac{N+2}{4}}\frac{\Delta Q(a_j)}{N}
\frac{a_{j,0}}{\big(\lambda^{(1)}_{j,\e}\big)^{\frac{N+2}{2}}}\int_{\R^N}|x|^2U^{\frac{N+2}{N-2}}_{0,1}(x)
\frac{\partial U_{0,\lambda}}{\partial \lambda}\big|_{\lambda=1}{\mathrm d}x
+o\Big(
\frac{1}{\bar \lambda_{\e}^{\frac{N+2}{2}}}   \Big)
\\ =&-\big(Q(a_j)\big)^{-\frac{N+2}{4}}\frac{(N-2)A}{N}
\frac{a_{j,0}}{\big(\lambda^{(1)}_{\e,j}\big)^{\frac{N+2}{2}}}\Delta Q(a_j)
+o\Big(
\frac{1}{\bar \lambda_{\e}^{\frac{N+2}{2}}}   \Big).
 \end{split}\end{equation*}
Here we use the fact
$$\int_{\R^N}|x|^2U^{\frac{N+2}{N-2}}_{0,1}(x)
\frac{\partial U_{0,\lambda}}{\partial \lambda}\big|_{\lambda=1}{\mathrm d}x=
\frac{1}{2^*}
\frac{\partial }{\partial \lambda}\Big( \int_{\R^N}|x|^2U^{2^*}_{0,\lambda}(x)
  {\mathrm d}x    \Big) \Big|_{\lambda=1} =-(N-2)A.
$$
Similarly,
\begin{equation*}
\begin{split}
 \int_{B_d(x^{(1)}_{\e,j})}&  D_{2,\varepsilon}\xi_\e {\mathrm d}x
=\big(Q(a_j)\big)^{-\frac{(N-2)s}{4}}
\frac{a_{j,0}}{\big(\lambda^{(1)}_{j,\e}\big)^{N-\frac{(N-2)s}{2}}}\int_{\R^N} U^{s}_{0,1}(x)
\frac{\partial U_{0,\lambda}}{\partial \lambda}\big|_{\lambda=1}{\mathrm d}x
+o\Big(
\frac{1}{\bar \lambda_{\e}^{N-\frac{(N-2)s}{2}}}\Big).
\end{split}
\end{equation*}
Also we know
\begin{equation*} \int_{\R^N} U^{s}_{0,1}(x)
\frac{\partial U_{0,\lambda}}{\partial \lambda}\big|_{\lambda=1}{\mathrm d}x=
\frac{1}{s+1}
\frac{\partial }{\partial \lambda}\Big( \int_{\R^N} U^{s+1}_{0,\lambda}(x)
  {\mathrm d}x    \Big) \Big|_{\lambda=1} =-\Big(\frac{N}{s+1}+1-\frac{N}{2}   \Big)B.
\end{equation*}
Next from \eqref{clp-10}, \eqref{2020-01-02-1} and \eqref{2020-01-02-2}, we find
\begin{equation*}
\begin{split}
\frac{1}{2^*}&\big(Q(a_j)\big)^{-\frac{N+2}{4}}\frac{(A+o(1))\Delta Q(a_j)}{\big(\lambda^{(1)}_{\e,j}\big)^{2}}=
-(1-\frac{N}{2}+\frac{N}{1+s})\big(Q(a_j)\big)^{-\frac{(N-2)s}{4}}
\frac{\varepsilon \big(B+o(1)\big)}{\big(\lambda^{(1)}_{j,\e}\big)^{2-\frac{(N-2)(s-1)}{2}}}.
\end{split}\end{equation*}
Hence we find
\begin{equation}\label{luo-131}
\mbox{LHS of \eqref{dclp-10}}=
\frac{(N-2)^2(1-s)}{4N}\big(Q(a_j)\big)^{-\frac{N+2}{4}}
\frac{Aa_{j,0}}{\big(\lambda^{(1)}_{\e,j}\big)^{\frac{N+2}{2}}}\Delta Q(a_j)
+o\Big(\frac{1}{\bar \lambda_{\e}^{\frac{N+2}{2}}}   \Big).
 \end{equation}
 On the other hand, from \eqref{ab4-18-12} and \eqref{3-3}, we have
 \begin{equation}\label{luo-132}
\begin{split}
\mbox{RHS of \eqref{dclp-10}}= O\Big(\frac{\log \bar \lambda_{\e}}{\bar \lambda_{\e}^{\frac{3(N-2)}{2}}}  \Big).
\end{split}\end{equation}
So we deduce \eqref{luo-13} by \eqref{luo-131} and \eqref{luo-132}.
\end{proof}

\begin{Rem}
Here we point out that we can't find \eqref{luo-13} when $N=4$. Because in this case, similar to \eqref{luo-131}, we can compute  that the main term of LHS of \eqref{dclp-10} is $0$. In fact
\begin{equation*}
\mbox{LHS of \eqref{dclp-10}}=O\Big(\frac{1}{\bar \lambda_{\e}^{3}}   \Big),~\mbox{and}~
\mbox{RHS of \eqref{dclp-10}}=O\Big( \frac{\log \lambda^{(1)}_{\e,j}}{\big(\lambda^{(1)}_{\e,j}\big)^{3}}\Big),
 \end{equation*}
which is not enough to get \eqref{luo-13}.
\end{Rem}

\begin{proof}[\textbf{Proof of Theorem \ref{th1.3}:}]
First, from \eqref{dd}, it holds,
\begin{equation*}
|\xi_{\varepsilon}(x)|=O\Big(\frac{1}{R^2}   \Big)+
O\Big(\frac{1}{R^{(N-2)(s-1)}}   \Big),~\mbox{for}~
x\in\Omega\backslash\bigcup_{j=1}^mB_{R(\lambda_{\varepsilon,j}^{(1)})^{-1}}(x_{\varepsilon,j}^{(1)}),
\end{equation*}
which implies that for any fixed $\gamma\in (0,1)$ and small $\varepsilon$, there exists $R_1>0$,
\begin{equation}\label{tian92}
|\xi_{\varepsilon}(x)|\leq \gamma,~ x\in\Omega\backslash\bigcup_{j=1}^mB_{R_1(\lambda_{\varepsilon,j}^{(1)})^{-1}}(x_{\varepsilon,j}^{(1)}).
\end{equation}
Also for the above fixed $R_1$, from \eqref{luo--13} and \eqref{luo-13}, we have
\begin{equation*}
\xi_{\varepsilon,j}(x)=o(1)~\mbox{in}~ B_{R_1}(0),~j=1,\cdots,m.
\end{equation*}
We know $\xi_{\varepsilon,j}(x)=\xi_{\varepsilon}(
\frac{x}{\lambda_{\varepsilon,j}^{(1)}}+x_{\varepsilon,j}^{(1)})$, so it follows
\begin{equation}\label{tian91}
\xi_{\varepsilon}(x)=o(1),~x\in \bigcup_{j=1}^m B_{R_1(\lambda_{\varepsilon,j}^{(1)})^{-1}}(x_{\varepsilon,j}^{(1)}).
\end{equation}
Hence for any fixed $\gamma\in (0,1)$ and small $\varepsilon$, \eqref{tian92} and \eqref{tian91} imply
$|\xi_{\varepsilon}(x)|\leq \gamma$ for all $x\in \Omega$,
which is in contradiction with $\|\xi_{\varepsilon}\|_{L^{\infty}(\Omega)}=1$. As a result, $u^{(1)}_{\varepsilon}(x)\equiv u^{(2)}_{\varepsilon}(x)$ for small $\varepsilon$.
\end{proof}

\appendix
\renewcommand{\theequation}{A.\arabic{equation}}

\setcounter{equation}{0}

\section{Proofs of various Pohozaev identities}
\setcounter{equation}{0}

\begin{proof}[\underline{\textbf{Proof of \eqref{clp-1}}}]
First  by  multiplying $\frac{\partial u_\varepsilon}{\partial x^i}$ on both sides of (1.1) and integrating on $\Omega'$, we have
\begin{equation}\label{1}
    \int_{\Omega'}-\Delta u_\varepsilon \frac{\partial u_\varepsilon}{\partial x^i}{\mathrm d}x=\int_{\Omega'}\Big(Q(x)u_\varepsilon^\frac{N+2}{N-2}+\varepsilon u_\varepsilon^{s}  \Big)\frac{\partial u_\varepsilon}{\partial x^i}{\mathrm d}x.
\end{equation}
Next, we have
\begin{equation}\label{2.7}
\begin{aligned}
\textrm{LHS of \eqref{1}}=&- \int_{\partial\Omega'}\frac{\partial u_\varepsilon(x)}{\partial x^i}\frac{\partial u_\varepsilon(x)}{\partial \nu}{\mathrm d}\sigma+ \int_{\Omega'}\nabla u_\varepsilon(x)\cdot\nabla\frac{\partial u_\varepsilon(x)}{\partial x^i}{\mathrm d}x \\ =&- \int_{\partial\Omega'}\frac{\partial u_\varepsilon(x)}{\partial x^i}\frac{\partial u_\varepsilon(x)}{\partial \nu}{\mathrm d}\sigma+\frac{1}{2}\int_{\partial\Omega'}|\nabla u_\varepsilon(x)|^2\nu^i(x){\mathrm d}\sigma.
\end{aligned}
\end{equation}
On the other hand, by Green's formula, we have
\begin{equation}\label{2.9}
\begin{split}
\textrm{RHS of \eqref{1}}=&\frac{1}{2^*}\int_{\Omega'} Q(x)\frac{\partial u^{2^*}_\varepsilon}{\partial x^i} {\mathrm d}x
+
\frac{\varepsilon}{s+1}\int_{\Omega'} \frac{\partial u_\varepsilon^{s+1}}{\partial x^i}{\mathrm d}x\\=&
 \frac{1}{2^*}\int_{\partial \Omega'} Q(x) u_\varepsilon^{2^*}\nu^i {\mathrm d}\sigma+ \frac{\varepsilon}{{s+1}}\int_{\partial \Omega'} u_\varepsilon^{s+1}\nu^i{\mathrm d}\sigma-
 \frac{1}{2^*}\int_{\Omega'} \frac{\partial Q(x)}{\partial x^i} u_\varepsilon^{2^*}{\mathrm d}x.
\end{split}\end{equation}
Then \eqref{clp-1} follows from   \eqref{2.7} and \eqref{2.9}.

\end{proof}
\begin{proof}[\underline{\textbf{Proof of \eqref{clp-10}}}]
Multiplying $(x-x_{\varepsilon,j})\cdot\nabla u_{\varepsilon}$ on both sides of \eqref{1.1} and integrating on $\Omega'$, we have
\begin{equation}\label{2019-11-24-01}
    \int_{\Omega'}-\Big((x-x_{\varepsilon,j})\cdot\nabla u_{\varepsilon}\Big)\Delta u_\varepsilon {\mathrm d}x=\int_{\Omega'}\Big((x-x_{\varepsilon,j})\cdot\nabla u_{\varepsilon}\Big) \Big(Q(x)u_\varepsilon^{\frac{N+2}{N-2}}+\varepsilon u_\varepsilon^{s}  \Big) {\mathrm d}x.
\end{equation}
First, we compute
\begin{equation}\label{2019-11-24-02}
\begin{split}
\mbox{RHS of \eqref{2019-11-24-01}}=&
\int_{\Omega'}\left(\frac{1}{2^*}\Big( \big(Q(x)(x-x_{\varepsilon,j})\big)\cdot \nabla u^{2^*}_{\varepsilon}\Big)+
\frac{\varepsilon}{s+1}\Big(\big( x-x_{\varepsilon,j}\big)\cdot\nabla u^{s+1}_{\varepsilon}\Big)
     \right){\mathrm d}x
\\ =&
\int_{\partial \Omega'} \Big(\frac{Q(x)}{2^*}u^{2^*}_\varepsilon+\frac{\varepsilon}{{s+1}
  } u_{\varepsilon}^{s+1}    \Big)  \Big(\big(x-x_{\varepsilon,j}\big)\cdot\nu\Big){\mathrm d}\sigma
\\  &
-\int_{\Omega'} \left(\frac{1}{2^*}\Big(  ( x-x_{\varepsilon,j}) \cdot \nabla Q(x) u^{2^*}_\varepsilon +NQ(x)   \Big)+\frac{N\varepsilon Q(x)}{{s+1}
  } u_{\varepsilon}^{s+1}   \right){\mathrm d}x.
\end{split}
\end{equation}
Next we find
\begin{equation}\label{2019-11-24-03}
\begin{split}
\mbox{LHS of \eqref{2019-11-24-01}}=&
\int_{\Omega'}\nabla u_\varepsilon \cdot \nabla\Big(\big(x-x_{\varepsilon,j}\big)\cdot\nabla u_{\varepsilon}   \Big){\mathrm d}x
-\int_{\partial\Omega'}
\Big(\big(x-x_{\varepsilon,j}\big)\cdot\nabla u_{\varepsilon}\Big)\frac{\partial u_\varepsilon}{\partial\nu}{\mathrm d}\sigma\\=&
\int_{\Omega'}|\nabla u_\varepsilon|^2 {\mathrm d}x+
\frac{1}{2}\int_{\Omega'} \Big(\big(x-x_{\varepsilon,j}\big)\cdot \nabla|\nabla u_\varepsilon|^2\Big){\mathrm d}x
-\int_{\partial\Omega'}\Big(\big(x-x_{\varepsilon,j}\big)\cdot \nabla u_{\varepsilon} \Big)\frac{\partial u_\varepsilon}{\partial\nu}
{\mathrm d}\sigma\\ =&
\frac{2-N}{2}\int_{\Omega'}|\nabla u_\varepsilon|^2 {\mathrm d}x+\frac{1}{2}
\int_{\partial \Omega'} \Big(\big( x-x_{\varepsilon,j}\big)\cdot\nu\Big)|\nabla u_\varepsilon|^2 {\mathrm d}\sigma
-\int_{\partial\Omega'}\Big(\big(x-x_{\varepsilon,j}\big)\cdot\nabla u_{\varepsilon}\Big) \frac{\partial u_\varepsilon}{\partial\nu}
{\mathrm d}\sigma.
\end{split}
\end{equation}
On the other hand, by  multiplying $u_\varepsilon$ on both sides of (1.1) and integrating on $\Omega'$, we get
\begin{equation}\label{2019-11-24-04}
\begin{split}
\int_{\Omega'}|\nabla u_\varepsilon|^2 {\mathrm d}x
=\int_{\partial\Omega'}\frac{\partial u_\varepsilon}{\partial\nu}
  u_{\varepsilon} {\mathrm d}\sigma+\int_{\Omega'}\Big(Q(x)u_\varepsilon^{2^*}+\varepsilon u_\varepsilon^{s+1}   \Big){\mathrm d}x.
\end{split}
\end{equation}
Hence, from \eqref{2019-11-24-02}, \eqref{2019-11-24-03} and \eqref{2019-11-24-04}, we find \eqref{clp-10}.
\end{proof}

\begin{proof}[\underline{\textbf{Proof of Claim 1}}]
Putting \eqref{luou} into  \eqref{a22-1-8} and \eqref{a2-13-1}, then
\begin{equation*}\label{14-2-1}
\sum_{q=1}^m\sum_{h=0}^N c_{\varepsilon, h, m}\Bigl\langle \frac{ \partial PU_{x_{\varepsilon, j}, \lambda_{\varepsilon, j}}}{\partial x^{ i}},
\frac{ \partial PU_{x_{\varepsilon, q,\lambda_{\varepsilon, q}}}}{\partial x^{ h}}
\Bigr\rangle_{\varepsilon}=0,\quad q=1,\cdots,m,\; i=0,1,\cdots, N,
\end{equation*}
here $x^0$ can be viewed as $\lambda$. Let the matrix
$$\textbf{M}_j:=\left(\displaystyle\sum_{q=1}^m \Bigl\langle \frac{ \partial PU_{x_{\varepsilon, j}, \lambda_{\varepsilon, j}}}{\partial x^{ i}},
\frac{ \partial PU_{x_{\varepsilon, q,\lambda_{\varepsilon, q}}}}{\partial x^{ h}}
\Bigr\rangle_{\varepsilon} \right)_{0\leq i,h\leq N}.$$
By the definition of
$PU_{x_{\varepsilon, j}, \lambda_{\varepsilon, j}}$, we can compute that the matrix
$\textbf{M}_j$ is  invertible and then  we conclude that
$c_{\varepsilon, i, j}=0,   i=1, \cdots, N,  ~j=1, \cdots, m.$

\end{proof}

\section{ Some useful estimates and Lemmas}

\renewcommand{\theequation}{B.\arabic{equation}}
Let
$\varphi_{x,\lambda}(y)=U_{x,\lambda}-PU_{x,\lambda}$,
then we know
\begin{align*}
 \begin{cases}
- \Delta  \varphi_{x,\lambda}(y)=0,   ~~& \text{in} ~\Omega,
 \\
\varphi_{x,\lambda}(y) = U_{x,\lambda}(y), ~~ & \text{on} ~\partial\Omega.
\end{cases}
\end{align*}
Thus we have  $  \varphi_{x,\lambda}(y)  >0$ in $\Omega$.

\begin{Lem}[c.f. \cite{Rey1990}]
For any fixed  $K\subset\subset \Omega$ and  $(x,\lambda)\in K\times \R^+$, we have
\begin{equation}\label{4-24-11}
\varphi_{x,\lambda}(y)=O\Big(\frac{1}{\lambda^{(N-2)/2}}   \Big), ~\mbox{uniformly for any}~ y\in  \Omega.
\end{equation}
\end{Lem}

\begin{Lem}[c.f. \cite{CPY19}]
For any $m\in \N^+$ and $p>1$, it holds
\begin{equation}\label{4-24-12}
\big(\sum^m_{i=1}a_i\big)^p- a_1^p
=O\Big(a_1^{p-1}\big(\sum^m_{i=2}a_i\big)+\sum^m_{i=2}a_i^p   \Big),
\end{equation}
   and
   \begin{equation}\label{3-14-11}
   \Big( \sum^m_{i=1}a_i\big)^p-\sum^m_{i=1}a_i^p
   =
   \begin{cases}
   O\Big(\displaystyle \sum_{i\neq j}a_i^{\frac{p}{2}}a_j^{\frac{p}{2}}   \Big),~&1<p<2,\\
   O\Big(\displaystyle \sum_{i\neq j}a_i^{p-1}a_j    \Big),~&p\geq 2.
   \end{cases}
   \end{equation}
   \end{Lem}

\begin{Lem}[c.f. \cite{Wei}, Lemma B.2]
For any constant $0<\theta\leq N-2$, there is a constant $C>0$, such that
 \begin{equation}\label{AB.2}
   \int_{\R^N}\frac{1}{|y-z|^{N-2}}\frac{1}{(1+|z|)^{2+\theta}}{\mathrm d}z\leq
   \begin{cases}
  {C\big(1+|y|\big)^{-\theta}},~& \theta< N-2,\\
    {C \big |\log |y| \big| \big(1+|y|\big)^{-\theta}},~& \theta=N-2.
   \end{cases}
 \end{equation}
\end{Lem} 

\begin{Lem}[Poincar\'e-Miranda theorem, c.f.\cite{Kulpa,Miranda}]\label{lem-B-1}
Consider $n$ continuous functions of
$n$ variables, $f_1,\cdots,f_n$. Assume that for each variable
$x_i$, the function $f_i$ is constantly negative when $x_i =-1$
 and constantly positive when $x_i = 1$. Then there is a point in the
$n$-dimensional cube $[-1,1]^{n}$
 in which all functions are simultaneously equal to $0$.
\end{Lem}
\vskip 0.5 cm

\noindent\textbf{Acknowledgments} This work was done while Lipeng Duan  was visiting
Department of Mathematical Sciences of University of Bath.    Lipeng Duan  is grateful to Professor Jun Yang   for many    insightful discussions.  
   Lipeng Duan was  supported 
by the China Scholarship Council and  NSFC grants (No.11771167). 
 Shuying Tian was partially supported by NSFC grants (No.12071364,11871387) and the Fundamental Research Funds for the Central Universities(No. WUT: 2020IA003).


\bibliography{reference}

\begin{thebibliography}{99}
{\footnotesize
\bibitem{Bartsch1} T. Bartsch, A. Micheletti, A. Pistoia, The Morse property for functions of Kirchhoff-Routh path type. \emph{Discrete Contin. Dyn. Syst. Series S}
 12, 1867--1877 (2019).
\bibitem{BN1983} H. Brezis, L. Nirenberg, Positive solutions of nonlinear elliptic equations involving critical Sobolev exponent. \emph{Comm. Pure Appl. Math.} 36, 437--478 (1983).

\bibitem{Cao1} D. Cao, S. Li, P. Luo,  Uniqueness of positive bound states with multi-bump for nonlinear Schr\"odinger equations. \emph{Calc. Var. Partial Differential Equations}  54, 4037--4063 (2015).

\bibitem{CN1995} D. Cao, E. Noussair, Multiple positive and nodal solutions for semilinear elliptic problems with critical exponents. \emph{Indiana Univ. Math. J.} 44, 1249--1271  (1995).
\bibitem{CPY19} D. Cao, S. Peng, S. Yan,  Singularly perturbed methods for nonlinear elliptic problems. Cambridge University Press, 2021, https://doi.org/10.1017/9781108872638.

\bibitem{CZ1997} D. Cao, X. Zhong, Multiplicity of positive solutions for semilinear elliptic equations involving the critical Sobolev exponents. \emph{Nonlin. Anal. TMA} 29, 461--483 (1997).

\bibitem{CY1999} J. Chabrowski, S. Yan, Concentration of solutions for a nonlinear elliptic problem with near critical exponent. \emph{Topological Methods in Nonlinear Analysis} 13, 199--233 (1999).

\bibitem{Deng}
Y. Deng, C. Lin, S. Yan, On the prescribed scalar curvature problem in $\R^N$, local uniqueness and periodicity. \emph{J. Math. Pures Appl.} 104, 1013--1044 (2015).

\bibitem{Esc} J. Escobar, Positive solutions for some semilinear elliptic equations with critical Sobolev exponents. \emph{Comm. Pure Appl. Math.} 40, 623--657 (1987).

\bibitem{Gilbarg}
D. Gilbarg, N. Trudinger, Elliptic partial differential equations of second order. Second edition. Springer-Verlag, Berlin, 1983.
\bibitem{Gla1993}
L. Glangetas, Uniqueness of positive solutions of a nonlinear elliptic equation involving
the critical exponent. \emph{Nonlin. Anal. TMA} 20, 571--603 (1993).

\bibitem{GMPY20}Y. Guo, M. Musso, S. Peng, S. Yan,
Non-degeneracy of multi-bubbling solutions for the prescribed scalar curvature equations and applications.
\emph{J. Funct. Anal.} 279, 108553 (2020).
\bibitem{Han}
Z. Han, Asymptotic approach to singular solutions for nonlinear elliptic equations involving critical Sobolev exponent. \emph{Ann. Inst. H. Poincar\'e Anal. Non Lin\'eaire}  8, 159--174 (1991).

\bibitem{Kulpa} W. Kulpa, The Poincar\'e-Miranda theorem. \emph{Am. Math. Mon.} 104, 545--550 (1997).

\bibitem{Miranda} C. Miranda,  Un'osservazione su un teorema di Brouwer. \emph{Boll. Un. Mat. Ital.} 3, 5--7
(1940).
\bibitem{MP2002} M. Musso, A. Pistoia, Multispike solutions for a nonlinear elliptic problem involving the critical Sobolev exponent. \emph{Indiana Univ. Math. J.}  51, 541--579   (2002).

\bibitem{PWY2018} S. Peng, C. Wang, S. Yan,
Construction of solutions via local Pohozaev identities. \emph{J. Funct. Anal.}  274, 2606--2633 (2018).
\bibitem{Rey1990}
O. Rey, The role of the Green's function in a nonlinear elliptic equation involving the critical Sobolev exponent. \emph{J. Funct. Anal.} 89, 1--52 (1990).
    \bibitem{Poh1965}
S. Pohozaev, On the eigenfunctions  of the equation $\Delta u+\lambda u=0$. (Russian) \emph{Dokl. Akad. Nauk}  165, 36--39  (1965).
\bibitem{Wei} J. Wei, S. Yan,  Infinitely many solutions for the prescribed scalar curvature problem on $S^N$. \emph{J. Funct. Anal.} 258, 3048--3081   (2010).
    }
\end{thebibliography}

\end{document}